\documentclass[12pt]{article}
\usepackage{appendix}
\usepackage{amsthm,bbm,amssymb,amsmath,bm,mathrsfs}
\usepackage{graphicx}
\usepackage{subfigure}
\usepackage{float}
\usepackage{epsfig}
\usepackage[colorlinks=true, citecolor=blue]{hyperref}
\usepackage{abstract}
\usepackage{natbib}
\usepackage{array}
\usepackage{booktabs}
\usepackage[english]{babel}
\usepackage[utf8]{inputenc}
\usepackage{algorithm,algpseudocode}  %for pseudo code writing

\allowdisplaybreaks[4]
\newtheorem{theorem}{Theorem}
\newtheorem{lemma}{Lemma}

\newtheorem{proposition}{Proposition}

\newtheorem{definition}{Definition}
\newtheorem{assumption}{Assumption}%[section]
\begin{document}
	
\title{\vspace{-2.6cm}
	Finite-Horizon Partially Observable Semi-Markov Games
	with  Risk Probability Criteria}
\author{{Xin Wen$^1$}, \ {Li Xia$^1$} \ {Zhihui Yu$^1$\thanks {yuzhh8@mail2.sysu.edu.cn\\This work has been supported by the National Key Research and Development Program of China (2022YFA1004600).}}  \\
	\small \vspace{-0.4cm}
	{$^1$School of Business, Sun Yat-Sen University, Guangzhou, China}\\
}
\date{}
\maketitle

\begin{abstract}
This paper studies partially observable two-person zero-sum semi-Markov games under a probability criterion, in which the 
system state may not be completely observed. 
It focuses on the probability that the accumulated rewards of player 1 (i.e., the incurred costs of player 2)  fall short of a specified target at the terminal stage, which represents the risk of player 1  and the capacity of player 2.
We study the game model via the technology of augmenting state space with the joint conditional distribution of the current unobserved state and the remaining  goal.
Under a mild condition, we establish a comparison theorem and derive the Shapley equation for the probability criterion. As a consequence, we prove the existence and the uniqueness of the value function and the existence of a Nash equilibrium. 
\end{abstract}


\section{Introduction}\label{sec1}
Stochastic game problems have received widespread attention in the literature, with the most commonly adopted optimality criteria being the (expected) discounted or average criterion \citep{shapley1953stochastic, sobel1971noncooperative, neyman2017continuous, babu2017stationary, chakrabarti2025stationary}. In these works, the dynamic programming principle holds, enabling the derivation of Shapley equations and the establishment of the existence of Nash equilibria. 
However, the expected optimality criteria are typically inadequate for reflecting the risk attitudes of decision-makers, thereby constraining its effectiveness in high-risk applications such as economic and managerial decision-making.
These limitations naturally call for a deeper exploration of risk criteria in stochastic games.

The exponential criterion is the most extensively studied risk criterion in stochastic games, known as risk-sensitive stochastic games. It employs an exponential utility function to measure the risk of accumulated rewards; see for instance, \cite{bauerle2017zero, golui2022continuous,ghosh2023discrete,zhang2025zero, chen2025zero,bhatt2025risk}. Thanks to the favorable mathematical properties of the exponential function, these works establish Shapley equations in product form, which facilitate further analysis. However, the exponential function often lacks practical interpretability and may fail to accurately reflect real-world risk. In contrast, probability serves as an effective risk measure and has been widely adopted in the fields of finance and management. Therefore, the probability criterion is a more worthwhile optimality criterion to study in the context of stochastic games,  which aims to minimize the probability that the accumulated rewards
 of player 1 (i.e., the incurred costs of player 2) fall short of a specified target within a decision horizon. Unfortunately, the probability metric is non-additive, rendering the conventional Shapley equation inapplicable to probabilistic stochastic games.  A commonly adopted remedy is the state augmentation technique, which incorporates the remaining target as an auxiliary state. This results in a Shapley equation defined over a two-dimensional state space, providing a tractable framework for analyzing the existence of the value function and a Nash equilibrium. Some excellent works can be referred to \cite{huang2017probability,bhabak2021continuous,huang2022zero}, and references therein, just to name a few.
 
%Recently, \cite{bhabak2021continuous} consider CTMGs for the  probability criterion  and show the existence of the value of the game and a randomized stationary saddle point. A two-person zero-sum SMGs  is discussed in \cite{huang2022zero},  and the existence condition of the optimal policy pair is proved and an approximate algorithm is proposed.

Most existing literature for stochastic games focuses on complete state information, where the state of the game is fully observed by all players at each stage. However, in many practical domains—such as artificial intelligence, maintenance, signal processing, and robotics—the system state is often only partially observable. For instance, in artificial intelligence, one of the most fundamental and challenging problems arises when agents must make decisions based on incomplete information about the underlying state and events in the environment. These limitations naturally motivate the study of partially observable stochastic games (POSGs), where players do not have full access to the true system state. Incorporating partial observability is essential for modeling decision-making in real-world scenarios more accurately.
The paper by \cite{ghosh2004zero} is one of the pioneering works on POSGs, which focuses on two-person zero-sum discrete-time POSGs with discounted criterion. By introducing control policies based on observations and the conditional distributions of unobservable states, the authors transform the original POSG into an equivalent completely observable stochastic game (COSG). By following this equivalence, many other works, such as \cite{ghosh2006partially, saha2014zero, liu2022sample, yan2024partially}, have extended the results of \cite{ghosh2004zero} to various other MDP settings.
The aforementioned works primarily focus on the expected discounted or expected average criteria, without addressing risk objectives. Only recently  have \cite{bhabak2024partially} and \cite{wei2025partially} begun exploring
two-person zero-sum discrete-time POSGs under risk criteria, where \cite{bhabak2024partially} consider the expected utility criterion and \cite{wei2025partially} adopt a probability criterion. The study of POSGs with risk criteria remains in its infancy  and is largely underexplored in the existing literature.

This paper studies two-person zero-sum  partially observable semi-Markov games (POSMGs) under a risk probability  criterion.  
The objective is to minimize the probability that the accumulated rewards of player 1 (or, equivalently, the incurred costs of player 2) within a finite horizon.  In contrast to the discrete-time POSGs studied in \cite{bhabak2024partially} and \cite{wei2025partially}, which assume fixed and deterministic sojourn times, POSMGs generalize the framework by allowing sojourn times to follow arbitrary probability distributions. This generalization enables the modeling of decision epochs as random event-driven time points, thereby capturing a broader class of real-world systems where transitions occur in a stochastic and irregular manner. However, the randomness of sojourn times renders existing solution techniques for discrete-time POSGs inapplicable to the finite-horizon POSMGs. Consequently, the POSMGs remains unsolved and poses a significant and timely research challenge.

We first formulate a fairly general model of controlled POSMGs with the risk probability criterion and assume that only one of two state components of a controlled semi-Markov process can be observed. However, the payoff functions and the transition probabilities may also depend on both components, which leads to the situation that the remaining  goal at each stage
is an unobservable quantity.
Referring to the expected discounted SMGs  in \cite{ghosh2006partially},  we proceed in a similar way for the risk probability criteria by enlarging the state space. We consider the POSMGs by introducing a recursive formula to generate the joint conditional distribution of the unobserved state and the remaining goal, and then the corresponding problem with complete observation can be reformulated.
Moreover, 
under a mild condition, we present a comparison theorem for the probability criterion, and establish the existence of a solution to the Shapley equation. Moreover, using the comparison theorem and  Shapley equation developed here, we also prove the existence of the value function and a Nash equilibrium of the game.

The remainder of this paper is organized as follows: In Section \ref{sec2}, we introduce the model of finite-horizon POSMGs under the probability criterion. In Section \ref{sec3}, we reformulate the semi-Markov game by defining a filter equation and expanding the state space. Additionally, we prove that the probability measure sequence generated by the filter equation represents the conditional distribution of the unobserved state and the remaining payoff goal. In Section \ref{sec4}, we establish the comparison theorem and the Shapley equation, prove the existence and uniqueness of the solution to the Shapley equations using the comparison theorem, and demonstrate the existence of the game value and a Nash equilibrium. Finally, in Section \ref{sec5}, we conclude the paper and discuss potential directions for future research.

\section{Partially Observable Semi-Markov Games}\label{sec2}
Let $\mathbb{N}_0:=\{0,1,\ldots\}$ be the set of non-negative integers, $\mathbb{R}:=(-\infty,+\infty)$ be the real number space, $\mathbb{R}_+:=[0,\infty)$ be the non-negative real number space, $\mathscr{B}(U)$ stands for the $\sigma$-algebra on any space $U$ and $\mathscr{P}(U)$ be the family of probability measures on the $\sigma$-algebra $\mathscr{B}(U)$.

The model of POSMGs is a controlled stochastic dynamic system consisting of objects
\begin{eqnarray}\label{model}
	\{E_X\times E_Y, A, B, A(x)\in A, B(x)\in B, Q, Q_0, r\}
\end{eqnarray}
$\bullet$  $E_X$ and $E_Y$ are  the observation and unobservation state spaces, respectively, and are both the Borel spaces;
\\$\bullet$ $A$ and $B$ denote the action spaces for player 1 and player 2, respectively, which are Borel spaces;
\\$\bullet$  $A(x)$ and $B(x)$ are the available action spaces when the system state $x\in E_X$, which are assumed to be finite.  We denote by $\mathbb{K}:=\{(x,y,a,b)\mid x\in E_X, y\in E_Y,a\in A(x), b\in B(x)\}$ the set of admissible state-action
pairs;
\\$\bullet$ $Q$ is the semi-Markov kernel on $\mathbb{R}_+\times E_X\times E_Y$ given $\mathbb{K}$, which describes the reansition mechanism of POSMG.
It is assumed to satisy that: (i) $Q(\cdot, D\mid x,y,a,b)$, for each given $D\in\mathscr{B}(E_X\times E_Y)$, $(x,y,a,b)\in\mathbb{K}$, is a nondecreasing and right-continuous real function on $\mathbb{R}_+$ such that $Q(0,D\mid x,y,a,b)=0$; (ii) for each fixed $t\in \mathbb{R}_+$, $Q(t,\cdot,\cdot\mid \cdot,\cdot,\cdot,\cdot)$ is a stochastic kernel on $E_X\times E_Y$ given $\mathbb{K}$; and (iii) $Q(\infty,\cdot,\cdot\mid \cdot,\cdot,\cdot,\cdot)=\lim\limits_{n\to\infty}Q(t,\cdot,\cdot\mid \cdot,\cdot,\cdot,\cdot)$ is a stochastic kernel on $E_X\times E_Y$ given $\mathbb{K}$;
If players 1 and 2 select actions $a\in A(x)$ and $b\in B(x)$ at state pair $(x,y)\in E_X\times E_Y$, respectively, then $Q(t,D\mid x,y,a,b)$ denotes the joint probability that the sojourn time in state $(x,y)$ is not greater than $t\in\mathbb{R}_+$ and the next pair is in $D\in\mathscr{B}(E_X\times E_Y)$;
\\$\bullet$ Moreover, $Q_0$ is the initial distribution of the unobservable state $y\in E_Y$;
\\$\bullet$ $r(x,y,a,b)$ is a nonnegative real-valued measurable function on $\mathbb{K}$, which represents the reward rate for player 1 (or the cost rate for player 2).

In what follows, we assume the semi-Markov kernel $Q$ has a measure density $q$ with respect to some $\sigma$-finite measures $\eta$ and $\mu$, i.e.,
\begin{eqnarray*}
	Q(t,D\mid x,y,a,b)=\int_{D}q(t,x',y'\mid x,y,a,b)\eta(dx')\mu(dy')\quad\forall ~t\in\mathbb{R}_+, B\in\mathscr{B}(E_X\times E_Y).
\end{eqnarray*}
For convenience, we denote by
\begin{eqnarray*}
	q^X(t,x'\mid x,y,a,b):=\int_{E_X}q(t,x',y'\mid x,y,a,b)\mu(dy')
\end{eqnarray*}
the marginal transition kernel density and we further define
\begin{eqnarray*}
	Q^X(t,D\mid x,{\mu},a,b):=\int_D\int_Eq^X(t,x'\mid,x,y,a,b){\mu}(dy)\eta(dx')\quad\forall ~D\in \mathscr{B}(E_X).
\end{eqnarray*}

Now we introduce the finite-horizon POSMGs with probability criteria.
Suppose that the initial decision time is $s_0:=0$ and the system state pair is in $(x_0,y_0)\in E_X\times E_Y$ where $x_0$ is observable and $y_0$ cannot be observed, and the players have a common goal $\lambda_0\in\mathbb{R}_+$ over a planning horizon $t_0\in\mathbb{R}_+$ (that is, within the planning horizon $t_0$, player 1 tries to earn reward more than $\lambda_0$, while player 2 wants to control the cost falls short of $\lambda_0$). Then the players independently select actions $a_0\in A(x_0)$ and $b_0\in B(x_0)$, respectively.
Consequently, the system remains in $(x_0,y_0)$ until time $s_1$ and transfers to a new state pair $(x_1,y_1)\in D\in \mathscr{B}(E_X\times E_Y)$ according to the transition probability $Q(\theta_1,D\mid x_0,y_0,a_0,b_0)$ where $\theta_1:=s_1-s_0$. At time $s_1$, a payoff $r(x_0,y_0,a_0,b_0)(\theta_1\wedge t_0)$ is earned and thus the current goal is $\lambda_1:=\lambda_0-r(x_0,y_0,a_0,b_0)(\theta_1\wedge t_0)$ over the remaining planning horizon $t_1:=[t_1-\theta_1]^+$ for two players, where $[y]^+:=\max\{y,0\}$. Then player 1 takes $a_1\in A(x_1)$ and player 2 chooses $b_1\in B(x_1)$, and the same sequence of events occurs.
The game evolves in this way.
At the $(n+1)$-th decision time $s_{n+1}$, the remianing planning horizon is $t_{n+1}:=[t_n-\theta_{n+1}]^+$ where $\theta_{n+1}=s_{n+1}-s_n$ and the remaining goal is $\lambda_{n+1}:=\lambda_n-r(x_n,y_n,a_n,b_n)(\theta_{n+1}\wedge t_n)$.

To construct a mathematical model of the game described above we introduce the probability space on which all the random variable are to be defined.
The canonical sample space is defined as $\Omega:=(\mathbb{R}_+\times \mathbb{R}_+\times E_X\times E_Y\times \mathbb{R}\times A\times B)^{\infty}$.
The random variables are defined as follows:
for each $\omega=(\theta_0,t_0,x_0,y_0,\lambda_0,a_0,b_0,\ldots, \theta_n,t_n,x_n,y_n,\lambda_n,a_n,b_n,\ldots)\in\Omega$,
let $\Theta_0(\omega)=\theta_0=0$, $\Theta_{n+1}(\omega)=\theta_{n+1}$, $X_n(\omega)=x_n$, $Y_n(\omega)=y_n$, $T_n(\omega)=t_n$, $\lambda_n(\omega)=\lambda_n$, $A_n(\omega)=a_n$, $B_n(\omega)=b_n$, $S_n(\omega)=s_n=\theta_1+\theta_2+\cdots+\theta_n$,
$h_n(\omega)=(\theta_0,t_0,x_0,y_0,\lambda_0,a_0,b_0,\ldots, \theta_n,t_n,x_n,y_n,\lambda_n)$,
$i_n(\omega)=(\theta_0,t_0,x_0,a_0,b_0,\ldots, \theta_n,t_n,x_n)$,
where $s_n$ is the $n$-th decision time, $\Theta_n$ is the sojourn time between the $(n-1)$-th and the $n$-th decision time, $X_n$, $Y_n$, $\Lambda_n$, $T_n$,
$A_n$ and $B_n$ represent the observable state, the unobservable state, the remaining goal, the remaining planning horizon, the actions for player 1 and player 2 at the $n$-th  decision epoch, respectively.

Moreover, $h_n$ denotes the entire history at $n$-th decision time which is  not available to the players for decision making and $i_n$ is the observed history  information.
Let $\mathscr{H}_n$ and $\mathscr{I}_n$ be the sets of all entire histories $h_n$ and all observable histories $i_n$ which are endowed with a Borel $\sigma$-field, respectively. Moreover, given the initial distribution of unobservable state $Q_0$ and the initial goal $\lambda$, we denote $Q_0(\cdot)\delta_{\lambda}(\cdot)$ as the initial joint distribution of  the unobservable state and the payoff goal. Without generality,
based on the initial joint distribution $\mu\in \mathscr{P}(E_Y\times \mathbb{R})$, and  the observed history $i_n\in  \mathscr{I}_n$, we define the following policies for the POSMGs, which represent the action-selection rules for the players.
\begin{definition}
	(a) An admissible policy for player 1 is a sequence $\pi^1=\{\pi_n^1,n\in\mathbb{N}_0\}$ of stochastic kernel $\pi_n^1$ on $A$
	given $\mathscr{P}(E_Y\times \mathbb{R}) \times \mathscr{I}_n$ such that $\pi_n^1(A(x_n)\mid \mu,i_n)=1$.
	The set of all admissible policies of player 1 is denoted by $\Pi^1$.\\
	(b) Similarly, we define the policy $\pi^2=\{\pi_n^2,n\in\mathbb{N}_0\}$ and the family $\Pi^2$ for player 2 with $A$ and $A(x_n)$ replaced by $B$ and $B(x_n)$ , respectively.
\end{definition}

For each $(t,x,\mu)\in \mathbb{R}_+\times E_X\times \mathscr{P}(E_Y\times \mathbb{R})$ and a pair of admissible policies $(\pi^1,\pi^2)\in \Pi^1\times \Pi^2$, by the well-known Ionescu-Tulcea's theorem, there exists a unique probability measure $\mathbb{P}^{\pi^1,\pi^2}_{(t,x,\mu)}$ on $(\Omega,\mathscr{B}(\Omega))$ satisfying the following construction:
$\mathbb{P}^{\pi^1,\pi^2}_{(t,x,\mu)}$ on $\mathscr{H}_0$ is given by $\mathbb{P}^{\pi^1,\pi^2}_{(t,x,\mu)}(0,t,dy,d\lambda)=\mu(dy,d\lambda)$.
Suppose that the measure $\mathbb{P}^{\pi^1,\pi^2}_{(t,x,\mu)}$ on $\mathscr{H}_n$ has been constructed, and then
$\mathbb{P}^{\pi^1,\pi^2}_{(t,x,\mu)}$  on $\mathscr{H}_{n+1}$ is determined by the following expression
\begin{eqnarray*}
	&&\mathbb{P}^{\pi^1,\pi^2}_{(t,x,\mu)}(\Gamma\times (da,db,ds,dt,dx',dy',d\lambda)):=\int_{\Gamma}\mathbb{P}^{\pi^1,\pi^2}_{(t,x,\mu)}(dh_n)\mathbb{I}_{\{t=[t_n-s]^+\}}\mathbb{I}_{\{\lambda=\lambda_n-r(x_n,y_n,a,b)(s\wedge t_n)\}}
	\nonumber\\
	&&~~~~~~~~~~~~~~~~~~~~~~~~~~~~~~~~~~~~~~~\times	Q(ds,dx',dy'\mid x_n,y_n,a,b)\pi_n^1(da\mid \mu, i_n)\pi_n^2(db\mid \mu,i_n)dtd\lambda,
\end{eqnarray*}
where $\mathbb{I}_C$ is the indicator function on any set $C$.
The expectation operator with respect to $\mathbb{P}^{\pi^1,\pi^2}_{(t,x,\mu)}$ is denoted by $\mathbb{E}^{\pi^1,\pi^2}_{(t,x,\mu)}$. 

From the construction of probability space $(\Omega,\mathscr{B}(\Omega),\mathbb{P}^{\pi^1,\pi^2}_{(t,x,\mu)})$, we define  the continuous time state-action pair, payoff goal processes $\{X(t),Y(t),A(t),B(t),\Lambda(t),t\in\mathbb{R}_+\}$ by
\begin{eqnarray*}
	&&X(t)=\sum_{n=0}^{\infty}X_n\mathbb{I}_{\{S_n\leq t<S_{n+1}\}}+X_{\infty}\mathbb{I}_{\{t\ge S_{\infty}\}},~~ Y(t)=\sum_{n=0}^{\infty}Y_n\mathbb{I}_{\{S_n\leq t<S_{n+1}\}}+Y_{\infty}\mathbb{I}_{\{t\ge S_{\infty}\}},\\
	&&A(t)=\sum_{n=0}^{\infty}A_n\mathbb{I}_{\{S_n\leq t<S_{n+1}\}}+A_{\infty}\mathbb{I}_{\{t\ge S_{\infty}\}},~~
	B(t)=\sum_{n=0}^{\infty}B_n\mathbb{I}_{\{S_n\leq t<S_{n+1}\}}+B_{\infty}\mathbb{I}_{\{t\ge S_{\infty}\}},\\
	&&\Lambda(t)=\sum_{n=0}^{\infty}\Lambda_n\mathbb{I}_{\{S_n\leq t<S_{n+1}\}}+\Lambda_{\infty}\mathbb{I}_{\{t\ge S_{\infty}\}},
\end{eqnarray*}
where $X_{\infty}\notin E_X$, $Y_{\infty}\notin E_Y$, $A_{\infty}\notin A$ and $B_{\infty}\notin B$ denote some isolated points and $S_{\infty}:=\lim_{n\to\infty}S_n$.

Then we consider the optimality problem about the POSMGs over finite time horizon.

For each initial $(t,x,\mu)\in \mathbb{R}_+\times E_X\times \mathscr{P}(E_Y\times \mathbb{R})$ and a pair of policies $(\pi^1,\pi^2)\in \Pi^1\times \Pi^2$, the risk probability criterion of POSMGs is defined by
\begin{eqnarray}\label{probability criterion}
	F^{\pi^1,\pi^2}(t,x,\mu):=\mathbb{P}^{\pi^1,\pi^2}_{(t,x,\mu)}(\int_0^t r(X(s),Y(s),A(s),B(s))ds\leq \Lambda(0)),
\end{eqnarray}
which represents the risk for player 1 that the rewards do not reach initial profit goal and also measures the capacity of player 2 for controlling the cost do not exceed the cost level.
Note that $t$, $x$ are the initial planning horizon and the initial state, respectively, and the measure $\mu\in \mathscr{P}(E_Y\times \mathbb{R})$ represents the joint distribution of the initial unobserved state and the initial payoff goal.

We also define the Nash equilibrium and the value of the game as follows.

\begin{definition}
	(a)	A policy $\pi^{*1}\in \Pi^1$ is said to be optimal for player 1 if
	\begin{eqnarray*}
		F^{\pi^{*1},\pi^2}(t,x,\mu)\leq \underline{V}(t,x,\mu)
		:=\sup_{\pi^2\in\Pi^2}\inf_{\pi^1\in\Pi^1}F^{\pi^1,\pi^2}(t,x,\mu)
		~~\forall  (t,x,\mu)\in \mathbb{R}_+\times E_X\times \mathscr{P}(E_X\times \mathbb{R})
	\end{eqnarray*}
	for any $\pi^2\in\Pi^2$.
	The function $\underline{V}$ is reffered to as the lower value of the POSMGs.
	\\(b) Similarly, a policy $\pi^{*2}\in \Pi^2$ is said to be optimal for player 2 if
	\begin{eqnarray*}
		F^{\pi^1,\pi^{*2}}(t,x,\mu)\ge \overline{V}(t,x,\mu)
		:=\inf_{\pi^1\in\Pi^1}\sup_{\pi^2\in\Pi^2}F^{\pi^1,\pi^2}(t,x,\mu)
		~~\forall  (t,x,\mu)\in \mathbb{R}_+\times E_X\times \mathscr{P}(E_X\times \mathbb{R})
	\end{eqnarray*}
	for any $\pi^1\in\Pi^1$.
	The function $\overline{V}$ is reffered to as the upper value of the POSMGs.
	\\Clearly, $\underline{V}(t,x,\mu)\leq \overline{V}(t,x,\mu)$ for every
	$ (t,x)\in \mathbb{R}_+\times E_X$.\\
	(c) If $\pi^{*k}\in\Pi^k$ is optimal for player $k$ ($k=1,2$), then
	$(\pi^{*1},\pi^{*2})$ is called a Nash equilibrium.\\
	(d) If
	$\underline{V}(t,x,\mu)=\overline{V}(t,x,\mu)$
	for every $  (t,x,\mu)\in \mathbb{R}_+\times E_X\times \mathscr{P}(E_X\times \mathbb{R})$, then the common function is called the value of the game, which is denoted by $V$.
\end{definition}

\section{Preliminaries}\label{sec3}

Our goal here is to give the conditions for the existence of Nash equilibrium. For this purpose, we will construct an updating procedure for a probability measure $\mu$ which can generate a sequence of the conditional distributions on the unobserved state and remaining payoff goal.
We first define the following updating operator
$\Upsilon: \mathbb{R}_+\times E_X\times \mathscr{P}(E_Y\times \mathbb{R})\times A\times B \times \mathbb{R}_+\times \mathbb{R}_+\times E_X\to \mathscr{P}(E_Y\times \mathbb{R}):$
\begin{eqnarray}
&&\Upsilon(t,x,\mu,a,b,\theta,t',x')(D)\nonumber\\
&=&\frac{\int_{E_Y}\int_{\mathbb{R}}\int_{D}q(\theta,x',y'\mid x,y, a, b)\nu(dy')\delta_{\{\lambda-r(x,y,a,b)(\theta\wedge t)\}}(d\lambda')\delta_{\{[t-\theta]^+\}}(t')\mu(dy,d\lambda)}{\int_{E_Y}q^X(\theta,x'\mid x,y,a,b)\mu^Y(dy)},
\end{eqnarray}
where $D\in\mathscr{B}(E_Y\times\mathbb{R})$, $\theta\in\mathbb{R}_+$ and $\mu^Y(dy):=\mu(dy,\mathbb{R})$ is the $Y$-marginal distribution of $\mu$.
Moreover, the $\Lambda$-marginal distribution of $\mu$ is given by $\mu^{\Lambda}(d\lambda):=\mu(E_Y,d\lambda)$.
In particular, we define the updating operator only when the denominator is positive.

For each $n\in\mathbb{N}_0$, $i_n=(\theta_0,t_0,x_0,a_0,b_0,\cdots,\theta_n,t_n,x_n)$, $D\in\mathscr{B}(E_Y\times \mathbb{R})$ and the initial joint distribution $\mu\in \mathscr{P}(E_Y\times\mathbb{R})$, we define a sequence of probability measures:
\begin{eqnarray}\label{filter}
&&\mu_0(D\mid i_0):=\mu(D),\nonumber\\
&&\mu_{n+1}(D\mid i_n,a,b,\theta,t,x'):=\Upsilon(t_n,x_n,\mu_n(\cdot\mid i_n),a,b,\theta,t,x')(D).
\end{eqnarray}
The recursion is called a filter equation.

 The next theorem show that the probability measures $\mu_n$ in (\ref{filter}) is a conditional distributions of unobserved state and the remaining payoff goal.

 \begin{theorem}\label{thm1}
 Suppose $(\mu_n)$ is given by the recursion (\ref{filter}).
 For each $n\in\mathbb{N}_0$, $\pi^1=\{\pi_n^1\}\in\Pi^1$, $\pi^2=\{\pi_n^2\}\in\Pi^2$, it holds that
 \begin{eqnarray*}
 &&\mu_n(D\mid \Theta_0,T_0,X_0,A_0,B_0,\cdots,\Theta_n,T_n,X_n)\\
 &=&\mathbb{P}^{\pi^1,\pi^2}_{(t,x,\mu)}((Y_n,\Lambda_n)\in D\mid \Theta_0,T_0,X_0,A_0,B_0,\cdots,\Theta_n,T_n,X_n)~~\forall~D\in\mathscr{B} (E_Y\times\mathbb{R}).
 \end{eqnarray*}
 \end{theorem}
\begin{proof}
We first prove that
\begin{eqnarray} \label{thm1-1}
&&\mathbb{E}^{\pi^1,\pi^2}_{(t,x,\mu)}[v(\Theta_0,T_0,X_0,A_0,B_0,\cdots,\Theta_n,T_n,X_n,Y_n,\Lambda_n)]\nonumber\\
&=&\mathbb{E}^{\pi^1,\pi^2}_{(t,x,\mu)}[v'(\Theta_0,T_0,X_0,A_0,B_0,\cdots,\Theta_n,T_n,X_n)],
\end{eqnarray}
for all bounded and measurable functions $v:\mathscr{I}_n\times E_Y\times \mathbb{R}\to\mathbb{R}$ and
\begin{eqnarray*}
v'(i_n):=\int_{E_Y}\int_{\mathbb{R}}v(i_n,y_n,\lambda_n)\mu_n(dy_n,d\lambda_n\mid i_n).
\end{eqnarray*}
Now we prove (\ref{thm1-1}) by induction.

For $n=0$, we obtain for the left-hand side that
\begin{eqnarray*}
\mathbb{E}^{\pi^1,\pi^2}_{(t,x,\mu)}[v(\Theta_0,T_0,X_0,Y_0,\Lambda_0)]=\int_{E_Y}\int_{\mathbb{R}}v(0,t,x,y,\lambda)\mu(dy,d\lambda),
\end{eqnarray*}
and for the right-hand side that
\begin{eqnarray*}
\mathbb{E}^{\pi^1,\pi^2}_{(t,x,\mu)}[v'(\Theta_0,T_0,X_0)]=\int_{E_Y}\int_{\mathbb{R}}v(0,t,x,y,\lambda)\mu(dy,d\lambda).
\end{eqnarray*}
Thus (\ref{thm1-1}) holds when $n=0$.
We assume the statement is true for $n-1$, that is,
\begin{eqnarray*}
&&\mathbb{E}^{\pi^1,\pi^2}_{(t,x,\mu)}[v(\Theta_0,T_0,X_0,A_0,B_0,\cdots,\Theta_{n-1},T_{n-1},X_{n-1},Y_{n-1},\Lambda_{n-1})]\\
&=&\mathbb{E}^{\pi^1,\pi^2}_{(t,x,\mu)}[v'(\Theta_0,T_0,X_0,A_0,B_0,\cdots,\Theta_{n-1},T_{n-1},X_{n-1})].
\end{eqnarray*}
That means
\begin{eqnarray*}
\mathbb{E}^{\pi^1,\pi^2}_{(t,x,\mu)}[v(i_{n-1},Y_{n-1},\Lambda_{n-1})]
=\mathbb{E}^{\pi^1,\pi^2}_{(t,x,\mu)}[\int_{E_Y}\int_{\mathbb{R}}v(i_{n-1},y_{n-1},\lambda_{n-1})\mu_{n-1}(dy_{n-1},d\lambda_{n-1}\mid i_{n-1})].
\end{eqnarray*}
Thus, for any given $i_{n-1}$, we have
\begin{eqnarray*}
&&\mathbb{E}^{\pi^1,\pi^2}_{(t,x,\mu)}[v(i_{n-1},A_{n-1},B_{n-1},\Theta_n,T_n,X_n,Y_n,\Lambda_n)]\\
&=&\int_{\mathscr{I}_{n-1}}di_{n-1}\int_{E_Y}\int_{\mathbb{R}_+}
\mu_{n-1}(dy_{n-1},d\lambda_{n-1}\mid i_{n-1})\\
&&~~
\sum_{a_{n-1}\in A(x_{n-1})}\sum_{b_{n-1}\in B(x_{n-1})}
\int_{E_X}\int_{E_Y}\int_{\mathbb{R}_+}
q(d\theta_n,x_n,y_n\mid x_{n-1},y_{n-1},a_{n-1},b_{n-1})\\
&&~~
\eta(dx_n)\nu(dy_n)
\pi^1_{n-1}(da_{n-1}\mid i_{n-1})\pi^2_{n-1}(db_{n-1}\mid i_{n-1})\\
&&~~
\int_{\mathbb{R}_+}\int_{\mathbb{R}}
v(i_{n-1},a_{n-1},b_{n-1},\theta_n,t_n,x_n,y_n,\lambda_n)\\
&&~~
\delta_{\{\lambda_{n-1}-r(x_{n-1},y_{n-1},a_{n-1},b_{n-1})(\theta_n\wedge t_{n-1})\}}(d\lambda_n)\delta_{\{[t_{n-1}-\theta_n]^+\}}(dt_n)\\
&=&\int_{\mathscr{I}_{n-1}}di_{n-1}
\int_{E_Y}\int_{\mathbb{R}_+}\mu_{n-1}(dy_{n-1},d\lambda_{n-1}\mid i_{n-1})\\
&&~~
\sum_{a_{n-1}\in A(x_{n-1})}\sum_{b_{n-1}\in B(x_{n-1})}
\int_{E_X}\int_{E_Y}\int_{\mathbb{R}_+}
q(d\theta_n,x_n,y_n\mid x_{n-1},y_{n-1},a_{n-1},b_{n-1})\\
&&~~
\eta(dx_n)\nu(dy_n)
\pi^1_{n-1}(da_{n-1}\mid i_{n-1})\pi^2_{n-1}(db_{n-1}\mid i_{n-1})\\
&&~~
v(i_{n-1},a_{n-1},b_{n-1},\theta_n,[t_{n-1}-\theta_n]^+,x_n,y_n,
\lambda_{n-1}-r(x_{n-1},y_{n-1},a_{n-1},b_{n-1})(\theta_n\wedge t_{n-1}))
\end{eqnarray*}
For the right-hand side, we obtain
\begin{eqnarray*}
&&\mathbb{E}^{\pi^1,\pi^2}_{(t,x,\mu)}[v'(i_{n-1},A_{n-1},B_{n-1},\Theta_n,T_n,X_n)]\\
&=&\int_{\mathscr{I}_{n-1}}di_{n-1}
\int_{E_Y}\int_{\mathbb{R}_+}\mu_{n-1}(dy_{n-1},d\lambda_{n-1}\mid i_{n-1})\\
&&~~
\sum_{a_{n-1}\in A(x_{n-1})}\sum_{b_{n-1}\in B(x_{n-1})}
\int_{E_X}\int_{\mathbb{R}_+}
q^X(d\theta_n,x_n\mid x_{n-1},y_{n-1},a_{n-1},b_{n-1})\\
&&~~
\eta(dx_n)
\pi^1_{n-1}(da_{n-1}\mid i_{n-1})\pi^2_{n-1}(db_{n-1}\mid i_{n-1})\\
&&~~
\int_{\mathbb{R}}v'(i_{n-1},a_{n-1},b_{n-1},\theta_n,t_n,x_n)
\delta_{\{[t_{n-1}-\theta_n]^+\}}(dt_n)\\
&=&\int_{\mathscr{I}_{n-1}}di_{n-1}
\int_{E_Y}\int_{\mathbb{R}_+}\mu_{n-1}(dy_{n-1},d\lambda_{n-1}\mid i_{n-1})\\
&&~~
\sum_{a_{n-1}\in A(x_{n-1})}\sum_{b_{n-1}\in B(x_{n-1})}
\int_{E_X}\int_{E_Y}\int_{\mathbb{R}_+}
q(d\theta_n,x_n,y_n\mid x_{n-1},y_{n-1},a_{n-1},b_{n-1})\\
&&~~
\eta(dx_n)\nu(dy_n)
\pi^1_{n-1}(da_{n-1}\mid i_{n-1})\pi^2_{n-1}(db_{n-1}\mid i_{n-1})\\
&&~~
\int_{E_Y}\int_{\mathbb{R}_+}
v(i_{n-1},a_{n-1},b_{n-1},\theta_n,[t_{n-1}-\theta_n]^+,x_n,y_n,
\lambda_n)\mu_n(dy_n,d\lambda_n\mid i_n)\\
&=&\int_{\mathscr{I}_{n-1}}di_{n-1}
\int_{E_Y}\int_{\mathbb{R}_+}\mu_{n-1}(dy_{n-1},d\lambda_{n-1}\mid i_{n-1})\\
&&~~
\sum_{a_{n-1}\in A(x_{n-1})}\sum_{b_{n-1}\in B(x_{n-1})}
\int_{E_X}\int_{E_Y}\int_{\mathbb{R}_+}
q(d\theta_n,x_n,y_n\mid x_{n-1},y_{n-1},a_{n-1},b_{n-1})\\
&&~~
\eta(dx_n)\nu(dy_n)
\pi^1_{n-1}(da_{n-1}\mid i_{n-1})\pi^2_{n-1}(db_{n-1}\mid i_{n-1})\\
&&~~
\int_{E_Y}\int_{\mathbb{R}_+}
v(i_{n-1},a_{n-1},b_{n-1},\theta_n,[t_{n-1}-\theta_n]^+,x_n,y_n,
\lambda_n)\\
&&\underline{\int_{E_Y}\int_{\mathbb{R}}
q(\theta_n,x_n,y_n\mid x_{n-1},y_{n-1}, a_{n-1}, b_{n-1})\nu(dy_n)}\\
&&
\frac{\delta_{\{\lambda_{n-1}-r(x_{n-1},y_{n-1},a_{n-1},b_{n-1})(\theta_n\wedge t_{n-1})\}}(d\lambda_n) \delta_{\{[t_{n-1}-\theta_n]^+\}}(t_n)\mu_{n-1}(dy_{n-1},d\lambda_{n-1}\mid i_{n-1})}
{\int_{E_Y}q^X(\theta_n,x_n\mid x_{n-1},y_{n-1},a_{n-1},b_{n-1})\mu_{n-1}^Y(dy_{n-1}\mid h_{n-1})}\\
%&&\frac{\int_{E_Y}\int_{\mathbb{R}}
%q(\theta_n,x_n,y_n\mid x_{n-1},y_{n-1}, a_{n-1}, b_{n-1})\nu(dy_n)
%\delta_{\{\lambda_{n-1}-r(x_{n-1},y_{n-1},a_{n-1},b_{n-1})(\theta_n\wedge t_{n-1})\}}(d\lambda_n) \delta_{\{[t_{n-1}-\theta_n]^+\}}(t_n)\mu_{n-1}(dy_{n-1},d\lambda_{n-1}\mid i_{n-1})}
%{\int_{E_Y}q^X(\theta_n,x_n\mid x_{n-1},y_{n-1},a_{n-1},b_{n-1})\mu_{n-1}^Y(dy_{n-1}\mid h_{n-1})}\\
&=&\int_{\mathscr{I}_{n-1}}di_{n-1}
\int_{E_Y}\int_{\mathbb{R}_+}\mu_{n-1}(dy_{n-1},d\lambda_{n-1}\mid i_{n-1})\\
&&~~
\sum_{a_{n-1}\in A(x_{n-1})}\sum_{b_{n-1}\in B(x_{n-1})}
\int_{E_X}\int_{E_Y}\int_{\mathbb{R}_+}
q(d\theta_n,x_n,y_n\mid x_{n-1},y_{n-1},a_{n-1},b_{n-1})\\
&&~~
\eta(dx_n)\nu(dy_n)
\pi^1_{n-1}(da_{n-1}\mid i_{n-1})\pi^2_{n-1}(db_{n-1}\mid i_{n-1})\\
&&~~v(i_{n-1},a_{n-1},b_{n-1},\theta_n,[t_{n-1}-\theta_n]^+,x_n,y_n,
\lambda_{n-1}-r(x_{n-1},y_{n-1},a_{n-1},b_{n-1})(\theta_n\wedge t_{n-1}))\\
&&\underline{\int_{E_Y}\int_{E_Y}\int_{\mathbb{R}}
q(\theta_n,x_n,y_n\mid x_{n-1},y_{n-1}, a_{n-1}, b_{n-1})\nu(dy_n)}\\
&&
\frac{\delta_{\{\lambda_{n-1}-r(x_{n-1},y_{n-1},a_{n-1},b_{n-1})(\theta_n\wedge t_{n-1})\}}(d\lambda_n) \delta_{\{[t_{n-1}-\theta_n]^+\}}(t_n)\mu_{n-1}(dy_{n-1},d\lambda_{n-1}\mid i_{n-1})}
{\int_{E_Y}q^X(\theta_n,x_n\mid x_{n-1},y_{n-1},a_{n-1},b_{n-1})\mu_{n-1}^Y(dy_{n-1}\mid h_{n-1})}\\
%&&\frac{\int_{E_Y}\int_{E_Y}\int_{\mathbb{R}}
%q(\theta_n,x_n,y_n\mid x_{n-1},y_{n-1}, a_{n-1}, b_{n-1})\nu(dy_n)
%\delta_{\{\lambda_{n-1}-r(x_{n-1},y_{n-1},a_{n-1},b_{n-1})(\theta_n\wedge t_{n-1})\}}(d\lambda_n) \delta_{\{[t_{n-1}-\theta_n]^+\}}(t_n)\mu_{n-1}(dy_{n-1},d\lambda_{n-1}\mid i_{n-1})}
%{\int_{E_Y}q^X(\theta_n,x_n\mid x_{n-1},y_{n-1},a_{n-1},b_{n-1})\mu_{n-1}^Y(dy_{n-1}\mid h_{n-1})}\\
&=&\int_{\mathscr{I}_{n-1}}di_{n-1}
\int_{E_Y}\int_{\mathbb{R}_+}\mu_{n-1}(dy_{n-1},d\lambda_{n-1}\mid i_{n-1})\\
&&~~
\sum_{a_{n-1}\in A(x_{n-1})}\sum_{b_{n-1}\in B(x_{n-1})}
\int_{E_X}\int_{E_Y}\int_{\mathbb{R}_+}
q(d\theta_n,x_n,y_n\mid x_{n-1},y_{n-1},a_{n-1},b_{n-1})\\
&&~~
\eta(dx_n)\nu(dy_n)
\pi^1_{n-1}(da_{n-1}\mid i_{n-1})\pi^2_{n-1}(db_{n-1}\mid i_{n-1})\\
&&~~
v(i_{n-1},a_{n-1},b_{n-1},\theta_n,[t_{n-1}-\theta_n]^+,x_n,y_n,
\lambda_{n-1}-r(x_{n-1},y_{n-1},a_{n-1},b_{n-1})(\theta_n\wedge t_{n-1})),
\end{eqnarray*}
which implies the results.
\end{proof}

To solve the optimality problem,
We also define a suitable model for semi-Markov games (SMGs) to solve POSMGs.

For this purpose we consider  SMGs with the state space $E:=E_X\times \mathscr{P}(E_Y\times\mathbb{R})$,
the action spaces $A$ and $B$.
$A(x,\mu):=A(x)$ and $B(x,\mu):=B(x)$ denote the sets of all actions available for player 1 and player 2, respectively, when the system is in state $(x,\mu)\in E$ and
$\tilde{\mathbb{K}}:=\{(x,\mu,a,b)\mid x\in E_X,\mu\in\mathscr{P}(E_Y\times\mathbb{R}),a\in A(x,\mu), b\in B(x,\mu)\}$ is the set of all measurable subsets of $E\times A\times B$.
For each
$(x,\mu)\in E$, $a\in A(x,\mu)$, $b\in B(x,\mu)$, $t\in\mathbb{R}_+$ and a Borel subset $B\subseteq E$, we  define the transition law  $\tilde{Q}$ by
\begin{eqnarray*}
	\tilde{Q}(t,B\mid x,\mu,a,b):=\int_{E_X}\mathbb{I}_B((x',\Upsilon(s,x,\mu,a,b,t,[s-t]^+,x')))
	Q^X(t,dx'\mid x,\mu^Y,a,b)\quad\forall s\in\mathbb{R}_+.
\end{eqnarray*}
Note that $\tilde{Q}$ is a semi-Markov kernel.
In the SMGs setting, based on the filter equation (\ref{filter}), we will introduce another policy set.
 \begin{definition}
(a) Let $\Phi$ denote the set of all stochastic kernel $\phi$ on $A$ given $\mathbb{R}_+\times E_X\times \mathscr{P}(E_Y\times \mathbb{R})$ such that $\phi(A(x)\mid t,x,\mu)=1$ for all $(t,x,\mu)\in \mathbb{R}_+\times E_X\times \mathscr{P}(E_Y\times \mathbb{R})$.\\
(b) A Markov policy of player 1 is a sequence $\pi^1=\{\phi_n,n\ge 0\}$ of stochastic kernel $\phi_n\in\Phi$.
 The set of all Markov policies for player 1 is denoted by $\Pi_M^1$.\\
 (c) A policy for player 1 is said to be stationary if there exists a stochastic kernel $\phi$ such that $\phi_n=\phi$ for all $n\in\mathbb{N}_0$. The set of all stationary policies for player 1 is denoted by $\Pi_S^1$.\\
(d) Similarly, with $B(x)$ in lieu of $A(x)$, we can define the following sets for player 2: the set $\Psi$ of all stochastic kernel $\psi$ on $B$ given $\mathbb{R}_+\times E_X\times \mathscr{P}(E_Y\times \mathbb{R})$ such that $\psi(B(x)\mid t,x,\mu)=1$, 
the sets $\Pi_M^2$ and $\Pi_S^2$ of Markov policies and stationary policies, respectively.
 \end{definition}
Note that we have $\Pi_M^1\subset\Pi^1$ and $\Pi_M^2\subset\Pi^2$ in the following sense: 
for every $\{\phi_n,n\ge 0\}\in \Pi_M^1$, we can fine a  $\{\pi_n^1,n\ge 0\}\in \Pi^1$ such that
\begin{eqnarray*}
&&\pi_0^1(\mu,i_0):=\phi_0(t_0,x_0,\mu)=\phi_0(t_0,x_0,\mu_0),\\
&&\pi_n^1(\mu,i_n):=\phi_n(t_n,x_n,\mu_n(\cdot\mid i_n)),~~n\ge 1,
\end{eqnarray*}
where $\{\mu_n\}$ satisfies the filter equation (\ref{filter}) and  $\mu\in \mathscr{P}(E_Y\times\mathbb{R})$ is the initial joint distribution.

Without loss of generality, we let horizon length be arbitrarily fixed with $T\in\mathbb{R}_+$.
In applications, it is natural to avoid the possibility of an infinite number of jumps during finite horizon.
For this purpose, we propose the following assumption (see Assumption 1 below) and a verification condition (see Proposition 1 below), whose proof can be similarly completed through the argument of Proposition 2.1 in \cite{huang2011finite}.
\begin{assumption}\label{ass1}
$\mathbb{P}^{\pi^1,\pi^2}_{(t,x,\mu)}(\{S_{\infty}>T\})=1$ for all $(t,x,\mu)\in \mathbb{R}_+\times E_X\times \mathscr{P}(E_Y\times \mathbb{R})$ and $(\pi^1,\pi^2)\in \Pi^1\times \Pi^2$.
\end{assumption}

\begin{proposition}
If there exists positive constants $\delta$ and $\epsilon$ such that
\begin{eqnarray*}
Q(\delta,E_X,E_Y\mid x,y,a,b)\leq 1-\epsilon ~~\forall~(x,y,a,b)\in \mathbb{K},
\end{eqnarray*}
then Assumption \ref{ass1} holds.
\end{proposition}

To calculate the probability $F^{\pi^1,\pi^2}(t,x,\mu)$, we have that for all $(t,x,\mu)\in \mathbb{R}_+\times E_X\times \mathscr{P}(E_Y\times \mathbb{R})$ and $(\pi^1,\pi^2)\in \Pi^1\times \Pi^2$,
\begin{eqnarray*} F^{\pi^1,\pi^2}(t,x,\mu)&:=&\mathbb{P}^{\pi^1,\pi^2}_{(t,x,\mu)}(\int_0^t r(X(s),Y(s),A(s),B(s))ds\leq \Lambda(0))\\
&=&\mathbb{P}^{\pi^1,\pi^2}_{(t,x,\mu)}(\sum_{m=0}^{\infty}\int_{S_m\wedge t}^{S_{m+1}\wedge t} r(X(s),Y(s),A(s),B(s))ds\leq \Lambda(0))\\
&=&\mathbb{P}^{\pi^1,\pi^2}_{(t,x,\mu)}(\bigcap_{n=0}^{\infty}\bigg\{\sum_{m=0}^{\infty}\int_{S_m\wedge t}^{S_{m+1}\wedge t} r(X(s),Y(s),A(s),B(s))ds\leq \Lambda_0\bigg\})\\
&=&\lim_{n\to\infty}\mathbb{P}^{\pi^1,\pi^2}_{(t,x,\mu)}(\sum_{m=0}^{n}\int_{S_m\wedge t}^{S_{m+1}\wedge t} r(X(s),Y(s),A(s),B(s))ds\leq \Lambda_0),
\end{eqnarray*}
where the second equality follows from Assumption \ref{ass1}, the fourth equation is due to the nonnegativity of the payoff function $r$ and the continuity of probability measures.

Moreover, we define the following function sequence.
\begin{eqnarray*}
F^{\pi^1,\pi^2}_{-1}(t,x,\mu)&:=&\int_{E_Y}\int_{\mathbb{R}}\mathbb{I}_{[0,\infty)}(\lambda)\mu(dy,d\lambda),\\
F^{\pi^1,\pi^2}_{n}(t,x,\mu)&:=&\mathbb{P}^{\pi^1,\pi^2}_{(t,x,\mu)}(\sum_{m=0}^{n}\int_{S_m\wedge t}^{S_{m+1}\wedge t} r(X(s),Y(s),A(s),B(s))ds\leq \Lambda_0).
\end{eqnarray*}
Clearly, $F^{\pi^1,\pi^2}_{n}(t,x,\mu)\ge F^{\pi^1,\pi^2}_{n+1}(t,x,\mu)$ for each $(t,x,\mu)\in \mathbb{R}_+\times E_X\times \mathscr{P}(E_Y\times \mathbb{R})$, $(\pi^1,\pi^2)\in \Pi^1\times \Pi^2$ and $n\ge -1$.
Furthermore, $\lim_{n\to\infty}F^{\pi^1,\pi^2}_{n}(t,x,\mu)=F^{\pi^1,\pi^2}(t,x,\mu)$.
\section{Comparison Theorem and Shapley Equation}\label{sec4}
Before illustrating our main results, we first introduce some notation:
let $\mathscr{F}_m$ be the set of Borel-measurable functions $F:\mathbb{R}_+\times E_X\times \mathscr{P}(E_X\times \mathbb{R})\to [0,1]$.
We also define operators $T^{\varphi_1,\varphi_2}$, $T$ and $T_{\pi_1,\pi_2}$ on $\mathscr{F}_m$ as follows:
for $H\in\mathscr{F}_m$, $(t,x,\mu)\in \mathbb{R}_+\times E_X\times \mathscr{P}(E_X\times \mathbb{R})$, $\varphi_1\in\mathscr{P}(A(x))$, $\varphi_2\in\mathscr{P}(B(x))$ and $(\pi_1,\pi_2)\in \Pi_S^1\times \Pi_S^2$:
\begin{eqnarray*}
T^{\varphi_1,\varphi_2}H(t,x,\mu)&=&\sum_{a\in A(x), b\in B(x)}\varphi_1(a)\varphi_2(b)\bigg[\int_{E_Y}\int_{\mathbb{R}}
\mathbb{I}_{[0,\lambda]}(r(x,y,a,b)t)(1-Q^X(t,E_X\mid x,y,a,b))\\
&&~~~~
\times\mu(dy,d\lambda)\\
&&~~+\int_{E_X}\int_0^t H(t-u,x',\Upsilon(t,x,\mu,a,b,u,t-u,x'))
Q^X(du,dx'\mid x,\mu^Y,a,b)\bigg];\\
T^{\pi_1,\pi_2}H(t,x,\mu)&=&T^{\pi^1(\cdot\mid t,x,\mu),\pi^2(\cdot\mid t,x,\mu)}H(t,x,\mu);\\
TH(t,x,\mu)&=&\inf_{\varphi_1\in\mathscr{P}(A(x))}\sup_{\varphi_2\in\mathscr{P}(B(x))}
T^{\varphi_1,\varphi_2}H(t,x,\mu).
\end{eqnarray*}
%Below we characterize the probability $F^{\pi_1,\pi_2}(t,x,\mu)$ by approximating.
%Before that, de denote the following notations.
Morover, 
for any $\pi^1=\{\phi_n,n\ge 0\}\in \Pi_M^1$ and $\pi^2=\{\psi_n,n\ge 0\}\in \Pi_M^2$, let $^{(k)}\pi^1:=\{\phi_{k+n},n\ge 0\}$ and $^{(k)}\pi^2:=\{\phi_{k+n},n\ge 0\}$ for each $k=0,1,\cdots$.
Obviously, we have $^{(k)}\pi^1\in \Pi_M^1$ and $^{(k)}\pi^2\in \Pi_M^2$, $^{(0)}\pi^1=\pi^1$ and $^{(0)}\pi^2=\pi^2$.

\begin{lemma}\label{lem1}
Under Assumption \ref{ass1}, for any $\pi^1=\{\phi_n,n\ge 0\}\in \Pi_M^1$ and $\pi^2=\{\psi_n,n\ge 0\}\in \Pi_M^2$, we have
\\(a) $F_n^{\pi^1,\pi^2}(\cdot,\cdot,\cdot)\in\mathscr{H}_m$,
$F^{\pi^1,\pi^2}(\cdot,\cdot,\cdot)\in\mathscr{H}_m$;
\\(b)
$F^{^{(k)}\pi^1,^{(k)}\pi^2}(t,x,\mu)
=T^{\phi_k,\psi_k}F^{^{(k+1)}\pi^1,^{(k+1)}\pi^2}(t,x,\mu)~
\forall k\ge 0, (t,x,\mu)\in \mathbb{R}_+\times E_X\times \mathscr{P}(E_X\times \mathbb{R})$;
\\(c)
$F^{\phi,\psi}(t,x,\mu)
=T^{\phi,\psi}F^{\phi,\psi}(t,x,\mu), \forall \phi\in\Pi_S^1, \psi\in\Pi_S^2.$;
\end{lemma}
\begin{proof}
(a) For any $(\pi^1,\pi^2)\in \Pi_M^1\times\Pi_M^2$, we have
\begin{eqnarray*}
&&F_{n+1}^{^{(k)}\pi^1,^{(k)}\pi^2}(t,x,\mu)
=\mathbb{P}^{^{(k)}\pi^1,^{(k)}\pi^2}_{(t,x,\mu)}(\sum_{m=0}^{n+1}\int_{S_m\wedge t}^{S_{m+1}\wedge t} r(X(s),Y(s),A(s),B(s))ds\leq \Lambda_0)\\
&=&\mathbb{P}^{^{(k)}\pi^1,^{(k)}\pi^2}_{(t,x,\mu)}\big[
\mathbb{I}_{\{\int_{0}^{t} r(X(s),Y(s),A(s),B(s))ds\leq \Lambda_0,S_1>t\}}\\
&&+\mathbb{I}_{\{\sum_{m=1}^{n+1}\int_{S_m\wedge t}^{S_{m+1}\wedge t} r(X(s),Y(s),A(s),B(s))ds\leq \Lambda_0-r(X_0,Y_0,A_0,B_0)(S_1-S_0),S_1\leq t\}}\big]\\
&=&\sum_{a\in A(x), b\in B(x)}\phi_k(a\mid t,x,\mu)\psi_k(b\mid t,x,\mu)\\
&&~
\bigg[\int_{E_Y}\int_{\mathbb{R}}
\mathbb{P}^{^{(k)}\pi^1,^{(k)}\pi^2}_{(t,x,\mu)}
(\int_{0}^{t} r(X(s),Y(s),A(s),B(s))ds\leq \Lambda_0,\\
&&~~S_1>t\mid S_0=0,T_0=t,X_0=x,Y_0=y,\Lambda_0=\lambda)\mu(dy,d\lambda)\\
&&+\int_{E_Y}\int_{\mathbb{R}}\int_{E_X}\int_0^t
\mathbb{P}^{^{(k)}\pi^1,^{(k)}\pi^2}_{(t,x,\mu)}
(\sum_{m=1}^{n+1}\int_{S_m\wedge t}^{S_{m+1}\wedge t} r(X(s),Y(s),A(s),B(s))ds\\
&&\leq \Lambda_0-r(X_0,Y_0,A_0,B_0)(S_1-S_0),S_1\leq t \mid S_0=0,T_0=t,\\
&&~~X_0=x,Y_0=y,\Lambda_0=\lambda,A_0=a,B_0=b,S_1=u,T_1=t-u,X_1=x')\\
&&~~Q^X(du,dx'\mid x,y,a,b)\mu(dy,d\lambda)\bigg]\\
&=&\sum_{a\in A(x), b\in B(x)}\phi_k(a\mid t,x,\mu)\psi_k(b\mid t,x,\mu)
\bigg[\int_{E_Y}\int_{\mathbb{R}}\mathbb{I}_{[0,\infty)}(\lambda-r(x,y,a,b)t)\\
&&~~(1-Q^X(t,E_X\mid x,y,a,b))\mu(dy,d\lambda)\\
&&+\int_{E_Y}\int_{\mathbb{R}}\int_{E_X}\int_{E_Y}\int_{\mathbb{R}}\int_0^t
\mathbb{P}^{^{(k)}\pi^1,^{(k)}\pi^2}_{(t,x,\mu)}
(\sum_{m=1}^{n+1}\int_{(S_m-u)\wedge (t-u)}^{(S_{m+1}-u)\wedge (t-u)} r(X(s+u),Y(s+u),\\
&&~~A(s+u),B(s+u))ds\leq \lambda-r(x,y,a,b)u\mid S_0=0,T_0=t,X_0=x,Y_0=y,\Lambda_0=\lambda,\\
&&A_0=a,B_0=b,S_1=u,T_1=t-u,X_1=x',Y_1=y',\Lambda_1=\lambda')\\
&&~~\delta_{\{\lambda-r(x,y,a,b)u\}}(d\lambda')\delta_{\{[t-u]^+\}}(t-u)Q(du,dx',dy'\mid x,y,a,b)\mu(dy,d\lambda)\bigg]\\
&=&\sum_{a\in A(x), b\in B(x)}\phi_k(a\mid t,x,\mu)\psi_k(b\mid t,x,\mu)
\bigg[\int_{E_Y}\int_{\mathbb{R}}\mathbb{I}_{[0,\infty)}(\lambda-r(x,y,a,b)t)\\
&&~~(1-Q^X(t,E_X\mid x,y,a,b))\mu(dy,d\lambda)\\
&&+\int_{E_Y}\int_{\mathbb{R}}\int_{E_X}\int_0^t
\mathbb{P}^{^{(k)}\pi^1,^{(k)}\pi^2}_{(t,x,\mu)}(\sum_{m=1}^{n+1}\int_{(S_m-u)\wedge (t-u)}^{(S_{m+1}-u)\wedge (t-u)} r(X(s+u),Y(s+u),\\
&&~~A(s+u),B(s+u))ds\leq \lambda-r(x,y,a,b)u\mid S_0=0,T_0=t,X_0=x,Y_0=y,\Lambda_0=\lambda,\\
&&A_0=a,B_0=b,S_1=u,T_1=t-u,X_1=x',Y_1=y',\Lambda_1=\lambda')\\
&&~~\int_{E_Y}\int_{\mathbb{R}}\delta_{\{\lambda-r(x,y,a,b)u\}}(d\lambda')
\delta_{\{[t-u]^+\}}(t-u)q(du,dx',dy'\mid x,y,a,b)\eta(dx')\nu(dy')\mu(dy,d\lambda)\bigg]\\
&=&\sum_{a\in A(x), b\in B(x)}\phi_k(a\mid t,x,\mu)\psi_k(b\mid t,x,\mu)
\bigg[\int_{E_Y}\int_{\mathbb{R}}\mathbb{I}_{[0,\infty)}(\lambda-r(x,y,a,b)t)\\
&&~~(1-Q^X(t,E_X\mid x,y,a,b))\mu(dy,d\lambda)\\
&&+\int_{E_Y}\int_{\mathbb{R}}\int_{E_X}\int_0^t
\mathbb{P}^{^{(k)}\pi^1,^{(k)}\pi^2}_{(t,x,\mu)}(\sum_{m=1}^{n+1}\int_{(S_m-u)\wedge (t-u)}^{(S_{m+1}-u)\wedge (t-u)} r(X(s+u),Y(s+u),\\
&&~~A(s+u),B(s+u))ds\leq \lambda-r(x,y,a,b)u\mid S_0=0,T_0=t,X_0=x,Y_0=y,\Lambda_0=\lambda,\\
&&A_0=a,B_0=b,S_1=u,T_1=t-u,X_1=x',Y_1=y',\Lambda_1=\lambda')\\
&&\frac{\int_{E_Y}\int_{\mathbb{R}}q(du,dx',dy'\mid x,y,a,b)\nu(dy')\delta_{\{\lambda-r(x,y,a,b)u\}}(d\lambda')
\delta_{\{[t-u]^+\}}(t-u)\mu(dy,d\lambda)}
{\int_{E_Y}q^X(u,x'\mid x,y,a,b)\mu^Y(dy)}\\
&&~~\times \int_{E_Y}q^X(du,x'\mid x,y,a,b)\mu^Y(dy)\eta(dx')\bigg]\\
&=&\sum_{a\in A(x), b\in B(x)}\phi_k(a\mid t,x,\mu)\psi_k(b\mid t,x,\mu)
\bigg[\int_{E_Y}\int_{\mathbb{R}}\mathbb{I}_{[0,\lambda]}(r(x,y,a,b)t)\\
&&~~(1-Q^X(t,E_X\mid x,y,a,b))\mu(dy,d\lambda)\\
&&+\int_{E_Y}\int_{\mathbb{R}}\int_{E_X}\int_0^t
\mathbb{P}^{^{(k)}\pi^1,^{(k)}\pi^2}_{(t,x,\mu)}
(\sum_{m=1}^{n+1}\int_{(S_m-u)\wedge (t-u)}^{(S_{m+1}-u)\wedge (t-u)} r(X(s+u),Y(s+u),\\
&&~~A(s+u),B(s+u))ds\leq \lambda'\mid S_0=0,T_0=t,X_0=x,Y_0=y,\Lambda_0=\lambda,\\
&&A_0=a,B_0=b,S_1=u,T_1=t-u,X_1=x',Y_1=y',\Lambda_1=\lambda')\\
&&\times \Upsilon(t,x,\mu,a,b,u,t-u,x')(dy',d\lambda')
Q^X(du,dx'\mid x,\mu^Y,a,b)\bigg]\\
&=&\sum_{a\in A(x), b\in B(x)}\phi_k(a\mid t,x,\mu)\psi_k(b\mid t,x,\mu)
\big[\int_{E_Y}\int_{\mathbb{R}}\mathbb{I}_{[0,\lambda]}(r(x,y,a,b)t)\\
&&~~(1-Q^X(t,E_X\mid x,y,a,b))\mu(dy,d\lambda)\\
&&+\mathbb{E}_{(t-u,x',\Upsilon)}^{^{(k+1)}\pi^1,^{(k+1)}\pi^2}
\bigg[\int_{E_X}\int_0^t\mathbb{P}_{(t-u,x',\Upsilon)}^{^{(k+1)}\pi^1,^{(k+1)}\pi^2}
(\sum_{m=0}^{n}\int_{(S_m)\wedge (t-u)}^{(S_{m+1})\wedge (t-u)} r(X(s),Y(s),\\
&&~~A(s),B(s))ds\leq \lambda'\mid S_0=0,T_0=t-u,X_0=x',Y_0=y',\Lambda_0=\lambda')\big]\\
&&\times Q^X(du,dx'\mid x,\mu^Y,a,b)\bigg]\\
&=&\sum_{a\in A(x), b\in B(x)}\phi_k(a\mid t,x,\mu)\psi_k(b\mid t,x,\mu)
\big[\int_{E_Y}\int_{\mathbb{R}}\mathbb{I}_{[0,\lambda]}(r(x,y,a,b)t)\\
&&~~(1-Q^X(t,E_X\mid x,y,a,b))\mu(dy,d\lambda)\\
&&+\int_{E_X}\int_0^t F_n^{^{(k+1)}\pi^1,^{(k+1)}\pi^2}
(t-u,x',\Upsilon(t,x,\mu,a,b,u,t-u,x'))Q^X(du,dx'\mid x,\mu^Y,a,b).
\end{eqnarray*}
Hence, for all $k=0,1,\cdots$, we have
\begin{eqnarray}\label{lem1-1}
F_{n+1}^{^{(k)}\pi^1,^{(k)}\pi^2}(t,x,\mu)
=T^{\phi_k,\psi_k}F_n^{^{(k+1)}\pi^1,^{(k+1)}\pi^2}
(t,x,\mu)
\end{eqnarray}
Then, letting $k=0$ in (\ref{lem1-1}), we can obtain the first assertion in part (a).
Moreover, since $\lim_{n\to\infty}F_{n}^{\pi^1,\pi^2}=F^{\pi^1,\pi^2}$,
we can complete this part.
\\(b) Taking $n\to\infty$ in (\ref{lem1-1}), by the dominated convergence theorem, we can get part (b).
\\(c) This statement is obviously true.
\end{proof}

\begin{lemma}\label{lem2}
 If a sequence $\{G_k,k\ge0\}$ of $G_k\in\mathscr{F}_m$ satisfies that for each $k\ge0$,
\begin{eqnarray}\label{lem2-1}
u(t,x,\mu)
=T^{\phi_k,\psi_k}u(t,x,\mu)~
\forall (t,x,\mu)\in \mathbb{R}_+\times E_X\times \mathscr{P}(E_X\times \mathbb{R}),
 \end{eqnarray}
then
$G_k(t,x,\mu)=F^{^{(k)}\pi^1,^{(k)}\pi^2}(t,x,\mu)$ for all $(t,x,\mu)\in \mathbb{R}_+\times E_X\times \mathscr{P}(E_X\times \mathbb{R})$ and $k\ge 0$.
\end{lemma}
\begin{proof} For convenience, we denote $t_{i+1}:=t_i-\theta_{i+1}$ if $t_i>\theta_{i+1}$ for all $i\ge 0$.
By (\ref{lem2-1}) we have for each $k\ge0$,
\begin{eqnarray*}
&&G_k(t_k,x_k,\mu_k)\\
&=&\sum_{a_k\in A(x_k), b_k\in B(x_k)}\phi_k(a_k\mid t_k,x_k,\mu_k)\psi_k(b_k\mid t_k,x_k,\mu_k)
\big[\int_{E_Y}\int_{\mathbb{R}}\mathbb{I}_{[0,\lambda_k]}(r(x_k,y_k,a_k,b_k)t_k)\\
&&~~(1-Q^X(t_k,E_X\mid x_k,y_k,a_k,b_k))\mu_k(dy_k,d\lambda_k)\\
&&+\int_{E_X}\int_0^{t_k} G_{k+1}
(t_k-\theta_{k+1},x_{k+1},
\Upsilon(t_k,x_k,\mu_k,a_k,b_k,\theta_{k+1},t_k-\theta_{k+1},x_{k+1}))\\
&&~~Q^X(d\theta_{k+1},dx_{k+1}\mid x_k,\mu_k^Y,a_k,b_k)\big]\\
&=&\mathbb{E}^{^{(k)}\pi^1,^{(k)}\pi^2}_{(t_k,x_k,\mu_k)}
\big[\mathbb{I}_{[0,\infty)}(\lambda_k -r(x_k,y_k,a_k,b_k)t_k)\mathbb{I}_{\{\theta_{k+1}>t_k\}}\big]\\
&&+\mathbb{E}^{^{(k)}\pi^1,^{(k)}\pi^2}_{(t_k,x_k,\mu_k)}
\big[G_{k+1}(t_k-\theta_{k+1},x_{k+1},\mu_{k+1})\mathbb{I}_{\{\theta_{k+1}\leq t_k\}}\big],
\end{eqnarray*}
where $\mu_{k+1}=\Upsilon(t_k,x_k,\mu_k,a_k,b_k,\theta_{k+1},t_k-\theta_{k+1},x_{k+1})$ mentioned above satisfying the filter equation (\ref{filter}) with $\mu_0=\mu$.
\\
Then, by $G_{k+1}(t_{k+1},x_{k+1},\mu_{k+1})=T^{\phi_{k+1}\psi_{k+1}}
G_{k+2}(t_{k+2},x_{k+2},\mu_{k+2})$, we can derive that
\begin{eqnarray}\label{lem2-2}
&&G_k(t_k,x_k,\mu_k)\nonumber\\
&=&\sum_{a_k\in A(x_k), b_k\in B(x_k)}\phi_k(a_k\mid t_k,x_k,\mu_k)\psi_k(b_k\mid t_k,x_k,\mu_k)\nonumber\\
&&~~\big[\int_{E_Y}\int_{\mathbb{R}}\mathbb{I}_{[0,\lambda_k]}(r(x_k,y_k,a_k,b_k)t_k)
(1-Q^X(t_k,E_X\mid x_k,y_k,a_k,b_k))\mu_k(dy_k,d\lambda_k)\nonumber\\
&&+\int_{E_X}\int_0^{t_k}
G_{k+1}(t_{k+1},x_{k+1},\mu_{k+1})
\mathbb{I}_{\{t_k-\theta_{k+1}\}}(t_{k+1})
Q^X(d\theta_{k+1},dx_{k+1}\mid x_k,\mu_k^Y,a_k,b_k)\big]\nonumber\\
&=&\mathbb{E}^{^{(k)}\pi^1,^{(k)}\pi^2}_{(t_k,x_k,\mu_k)}
\big[\mathbb{I}_{[0,\infty)}(\lambda_k -r(x_k,y_k,a_k,b_k)t_k)\mathbb{I}_{\{\theta_{k+1}>t_k\}}\big]\nonumber\\
&&+\sum_{a_k\in A(x_k), b_k\in B(x_k)}\phi_k(a_k\mid t_k,x_k,\mu_k)\psi_k(b_k\mid t_k,x_k,\mu_k)\int_{E_X}\int_0^{t_k}\nonumber\\
&&\sum_{a_{k+1}\in A(x_{k+1}), b_{k+1}\in B(x_{k+1})}\phi_{k+1}(a_{k+1}\mid t_{k+1},x_{k+1},\mu_{k+1})\psi_{k+1}(b_{k+1}\mid t_{k+1},x_{k+1},\mu_{k+1})\nonumber\\
&&~~\int_{E_Y}\int_{\mathbb{R}}
\mathbb{I}_{[0,\lambda_{k+1}]}(r(x_{k+1},y_{k+1},a_{k+1},b_{k+1})t_{k+1})\nonumber\\
&&(1-Q^X(t_{k+1},E_X\mid x_{k+1},y_{k+1},a_{k+1},b_{k+1}))\mu_{k+1}(dy_{k+1},d\lambda_{k+1})\nonumber\\
&&\mathbb{I}_{\{t_k-\theta_{k+1}\}}(t_{k+1})
Q^X(d\theta_{k+1},dx_{k+1}\mid x_k,\mu_k^Y,a_k,b_k)\nonumber\\
&&+\sum_{a_k\in A(x_k), b_k\in B(x_k)}\phi_k(a_k\mid t_k,x_k,\mu_k)\psi_k(b_k\mid t_k,x_k,\mu_k)\int_{E_X}\int_0^{t_k}\nonumber\\
&&\sum_{a_{k+1}\in A(x_{k+1}), b_{k+1}\in B(x_{k+1})}\phi_{k+1}(a_{k+1}\mid t_{k+1},x_{k+1},\mu_{k+1})\psi_{k+1}(b_{k+1}\mid t_{k+1},x_{k+1},\mu_{k+1})\nonumber\\
&&\int_{E_X}\int_0^{t_{k+1}}G_{k+2}(t_{k+2},x_{k+2},\mu_{k+2})
\mathbb{I}_{\{t_{k+1}-\theta_{k+2}\}}(t_{k+2})
Q^X(d\theta_{k+2},dx_{k+2}\mid x_{k+1},\mu_{k+1}^Y,a_{k+1},b_{k+1})\nonumber\\
&&\mathbb{I}_{\{t_k-\theta_{k+1}\}}(t_{k+1})
Q^X(d\theta_{k+1},dx_{k+1}\mid x_k,\mu_k^Y,a_k,b_k)\nonumber\\
&=&\mathbb{E}^{^{(k)}\pi^1,^{(k)}\pi^2}_{(t_k,x_k,\mu_k)}
\big[\mathbb{I}_{[0,\infty)}(\lambda_k -r(x_k,y_k,a_k,b_k)t_k)\mathbb{I}_{\{\theta_{k+1}>t_k\}}\big]\nonumber\\
&&+\mathbb{E}^{^{(k)}\pi^1,^{(k)}\pi^2}_{(t_k,x_k,\mu_k)}
\big[\mathbb{I}_{[0,\infty)}(\lambda_k -r(x_{k},y_{k},a_{k},b_{k})\theta_{k+1}
-r(x_{k+1},y_{k+1},a_{k+1},b_{k+1})(t_k-\theta_{k+1}))\nonumber\\
&&\mathbb{I}_{\{\theta_{k+1}<t_k,\theta_{k+2}>t_k-\theta_{k+1}\}}\big]\nonumber\\
&&+\mathbb{E}^{^{(k)}\pi^1,^{(k)}\pi^2}_{(t_k,x_k,\mu_k)}
\big[G_{k+2}(t_{k+2},x_{k+2},\mu_{k+2})\mathbb{I}_{\{\theta_{k+1}+\theta_{k+2}\leq t_k\}}\big]\nonumber\\
&=&\mathbb{E}^{^{(k)}\pi^1,^{(k)}\pi^2}_{(t_k,x_k,\mu_k)}
\big[\mathbb{I}_{[0,\lambda_k]}( \int_0^{t_k}r(X(s),Y(s),A(s),B(s))ds)
\mathbb{I}_{\{\theta_{k+1}+\theta_{k+2}>t_k\}}\big]\nonumber\\
&&+\mathbb{E}^{^{(k)}\pi^1,^{(k)}\pi^2}_{(t_k,x_k,\mu_k)}
\big[G_{k+2}(t_{k+2},x_{k+2},\mu_{k+2})\mathbb{I}_{\{\theta_{k+1}+\theta_{k+2}\leq t_k\}}\big].
\end{eqnarray}
Iterating (\ref{lem2-2}) $(N-k)$ times and we arrive at the following equation
\begin{eqnarray*}
G_k(t_k,x_k,\mu_k)
&=&\mathbb{E}^{^{(k)}\pi^1,^{(k)}\pi^2}_{(t_k,x_k,\mu_k)}
\big[\mathbb{I}_{[0,\lambda_k]}( \int_0^{t_k}r(X(s),Y(s),A(s),B(s))ds)
\mathbb{I}_{\{\sum_{m=1}^{N-k}\theta_{k+m}>t_k\}}\big]\nonumber\\
&&+\mathbb{E}^{^{(k)}\pi^1,^{(k)}\pi^2}_{(t_k,x_k,\mu_k)}
\big[G_N(t_k-\sum_{m=1}^{N-k-1}\theta_{k+m},x_N,\mu_N)
\mathbb{I}_{\{\sum_{m=1}^{N-k}\theta_{k+m}\leq t_k\}}\big].
\end{eqnarray*}
Let $N\to\infty$, by Assumption \ref{ass1} and the dominated convergence theorem, we have
\begin{eqnarray*}
G_k(t,x,\mu)&=&\lim_{N\to\infty}
\big(\mathbb{E}^{^{(k)}\pi^1,^{(k)}\pi^2}_{(t_k,x_k,\mu_k)}
\big[\mathbb{I}_{[0,\lambda]}( \int_0^{t}r(X(s),Y(s),A(s),B(s))ds)
\mathbb{I}_{\{S_N>t\}}\big]\nonumber\\
&&+\mathbb{E}^{^{(k)}\pi^1,^{(k)}\pi^2}_{(t_k,x_k,\mu_k)}
\big[G_N(t-\sum_{m=1}^{N}\theta_{m},x_N,\mu_N)
\mathbb{I}_{\{S_N\leq t\}}\big]\big)\\
&=&\mathbb{P}^{^{(k)}\pi^1,^{(k)}\pi^2}_{(t_k,x_k,\mu_k)}
\big( \int_0^{t}r(X(s),Y(s),A(s),B(s))ds\leq \Lambda_0)\nonumber\\
&=&F^{^{(k)}\pi^1,^{(k)}\pi^2}(t,x,\mu).
\end{eqnarray*}
Thus, Lemma \ref{lem2} has been proved.
\end{proof}

\begin{theorem}\label{Comparison}
(Comparison Theorem) Under Assumption \ref{ass1}, the following assertions hold.
\\(a) For any fixed  $\pi^{2}=\{\psi_k\}\in\Pi^2_M$,  if a sequence $\{G_k,k\ge0\}$ of $G_k\in\mathscr{F}_m$ satisfies that
\begin{eqnarray*}
&&G_k(t,x,\mu)\\
&=&\inf_{\phi\in\mathscr{P}(A(x))}\bigg\{
\sum_{a\in A(x), b\in B(x)}\phi(a)\psi_k(b\mid x,\mu)\big[\int_{E_Y}\int_{\mathbb{R}}
\mathbb{I}_{[0,\lambda]}(r(x,y,a,b)t)(1-Q^X(t,E_X\mid x,y,a,b))\nonumber\\
&&\times\mu(dy,d\lambda)+\int_{E_X}\int_0^t G_{k+1}(t-\theta,x',\Upsilon(t,x,\mu,a,b,\theta,t-\theta,x'))
Q^X(du,dx'\mid x,\mu^Y,a,b)\big]\bigg\},
\end{eqnarray*}
for all $(t,x,\mu)\in \mathbb{R}_+\times E_X\times \mathscr{P}(E_X\times \mathbb{R})$ and $k\ge 0$,  then $G_k(t,x,\mu)\leq F^{\pi^1,^{(k)}\pi^2}(t,x,\mu)$ $\forall (t,x,\mu)\in \mathbb{R}_+\times E_X\times \mathscr{P}(E_X\times \mathbb{R})$, $k\ge 0$ and $\pi^1\in\Pi^1$.
\\(b) For any fixed  $\pi^{1}=\{\phi_k\}\in\Pi^1_M$,  if a sequence $\{G_k,k\ge0\}$ of $G_k\in\mathscr{F}_m$ satisfies that
\begin{eqnarray*}
	&&G_k(t,x,\mu)\\
	&=&\sup_{\psi\in\mathscr{P}(B(x))}\bigg\{
	\sum_{a\in A(x), b\in B(x)}\phi_k(a\mid x,\mu)\psi(b)\big[\int_{E_Y}\int_{\mathbb{R}}
	\mathbb{I}_{[0,\lambda]}(r(x,y,a,b)t)(1-Q^X(t,E_X\mid x,y,a,b))\nonumber\\
	&&\times\mu(dy,d\lambda)+\int_{E_X}\int_0^t G_{k+1}(t-\theta,x',\Upsilon(t,x,\mu,a,b,\theta,t-\theta,x'))
	Q^X(du,dx'\mid x,\mu^Y,a,b)\big]\bigg\},
\end{eqnarray*}
for all $(t,x,\mu)\in \mathbb{R}_+\times E_X\times \mathscr{P}(E_X\times \mathbb{R})$ and $k\ge 0$, then $G_k(t,x,\mu)\ge F^{^{(k)}\pi^1,\pi^2}(t,x,\mu)$ $\forall (t,x,\mu)\in \mathbb{R}_+\times E_X\times \mathscr{P}(E_X\times \mathbb{R})$, $k\ge 0$ and $\pi^2\in\Pi^2$.
\end{theorem}
\begin{proof}
(a) For any fixed  $\pi^{2}=\{\psi_k\}\in\Pi^2_M$ and arbitrary $\pi_1\in\Pi^1$, using the similar technique as Lemma \ref{lem2}, we obtain
\begin{eqnarray*}
G_k(t_k,x_k,\mu_k)
&\leq&\mathbb{E}^{\pi^1,^{(k)}\pi^2}_{(t_k,x_k,\mu_k)}
\big[\mathbb{I}_{[0,\infty)}(\lambda_k- \sum_{m=k}^{N}r(x_m,y_m,f_m,g_m)(\theta_{m+1}\wedge t_m))
\mathbb{I}_{\{\sum_{m=k}^{N+k}\theta_{m+1}>t_k\}}\big]\nonumber\\
&&+\mathbb{E}^{\pi^1,^{(k)}\pi^2}_{(t_k,x_k,\mu_k)}
\big[G_{N+k}(t_k-\sum_{m=k}^{N+k}\theta_{m},x_{N+k},\mu_{N+k})
\mathbb{I}_{\{\sum_{m=k}^{N+k}\theta_{m+1}\leq t_k\}}\big].
\end{eqnarray*}
As the proof in Lemma \ref{lem2}, we get $G_k(t,x,\mu)\leq F^{\pi^1,^{(k)}\pi^2}(t,x,\mu)$ for all $(t,x,\mu)\in \mathbb{R}_+\times E_X\times \mathscr{P}(E_X\times \mathbb{R})$, $k\ge 0$ and $\pi^1\in\Pi^1$.
\\(b) It can be proved similarly.
\end{proof}

Next we  show the existence of the value of the game and Nash equilibrium.
\begin{theorem}\label{thm2}
(a) Suppose that Assumption \ref{ass1} holds.
Then the game has a value which is the unique solution satisfying the Shapley equation (\ref{thm2-1}); that is, for every $(t,x,\mu)\in \mathbb{R}_+\times E_X\times \mathscr{P}(E_X\times \mathbb{R})$,
\begin{eqnarray}\label{thm2-1}
V(t,x,\mu)&=&\inf_{\phi\in\mathscr{P}(A(x))}\sup_{\psi\in\mathscr{P}(B(x))}
\bigg\{\sum_{a\in A(x),b\in B(x)}\bigg[\int_{E_Y}\int_{\mathbb{R}}
\mathbb{I}_{[0,\lambda]}(r(x,y,a,b)t)\nonumber\\
&&~~(1-Q^X(t,E_X\mid x,y,a,b))\mu(dy,d\lambda)\nonumber\\
&&+\int_{E_X}\int_0^t V(t-\theta,x',\Upsilon(t,x,\mu,a,b,\theta,t-\theta,x'))
Q^X(du,dx'\mid x,\mu^Y,a,b)\bigg]\phi(a)\psi(b)\bigg\},\nonumber\\
\end{eqnarray}
(b) There exists a Nash equilibrium $(\phi^*,\psi^*)\in \Pi_S^1\times \Pi_S^2$ satisfying $\forall (t,x,\mu)\in \mathbb{R}_+\times E_X\times \mathscr{P}(E_X\times \mathbb{R})$,
\begin{eqnarray*}
V(t,x,\mu)=T^{\phi^*,\psi^*}V(t,x,\mu)
=\min_{\phi\in\mathscr{P}(A(x))}T^{\phi,\psi^*}V(t,x,\mu)
=\max_{\psi\in\mathscr{P}(B(x))}T^{\phi^*,\psi}V(t,x,\mu).
\end{eqnarray*}
\end{theorem}
\begin{proof}
Fix arbitrary $(t,x,\mu)\in \mathbb{R}_+\times E_X\times \mathscr{P}(E_X\times \mathbb{R})$ and each integer $k\ge -1$, we define
\begin{eqnarray}\label{thm2-2}
u^{-1}(t,x,\mu)&:=&\int_{\mathbb{R}}
\mathbb{I}_{[0,\infty)}(\lambda)\mu^{\Lambda}(d\lambda),\nonumber\\
u^{0}(t,x,\mu)&:=&\inf_{\phi\in\mathscr{P}(A(x))}\sup_{\psi\in\mathscr{P}(B(x))}
\bigg\{\sum_{a\in A(x),b\in B(x)}\bigg[\int_{E_Y}\int_{\mathbb{R}}
\mathbb{I}_{[0,\lambda]}(r(x,y,a,b)t)\nonumber\\
&&~~(1-Q^X(t,E_X\mid x,y,a,b))\mu(dy,d\lambda)\nonumber\\
&&+\int_{E_X}\int_0^t u^{-1}(t-\theta,x',\Upsilon(t,x,\mu,a,b,\theta,t-\theta,x'))
Q^X(d\theta,dx'\mid x,\mu^Y,a,b)\bigg]\phi(a)\psi(b)\bigg\},\nonumber\\
&&~~~\cdots\nonumber\\
&&~~~\cdots\nonumber\\
%&&~~~\cdots\nonumber\\
u^{k+1}(t,x,\mu)&:=&\inf_{\phi\in\mathscr{P}(A(x))}\sup_{\psi\in\mathscr{P}(B(x))}
\bigg\{\sum_{a\in A(x),b\in B(x)}\bigg[\int_{E_Y}\int_{\mathbb{R}}
\mathbb{I}_{[0,\lambda]}(r(x,y,a,b)t)\nonumber\\
&&~~(1-Q^X(t,E_X\mid x,y,a,b))\mu(dy,d\lambda)\nonumber\\
&&+\int_{E_X}\int_0^t u^{k}(t-\theta,x',\Upsilon(t,x,\mu,a,b,\theta,t-\theta,x'))
Q^X(d\theta,dx'\mid x,\mu^Y,a,b)\bigg]\phi(a)\psi(b)\bigg\},\nonumber\\
\end{eqnarray}
Now we claim that for any fixed $t\ge -1$, $u^t(\cdot,\cdot,\cdot)\in\mathscr{F}_m$ for all $t\ge -1$.
\\It is clear that $u^{-1}(t,x,\mu)=\int_{\mathbb{R}}
\mathbb{I}_{[0,\infty)}(\lambda)\mu^{\Lambda}(d\lambda)\in\mathscr{F}_m$.
We assume that $u^k(\cdot,\cdot,\cdot)\in\mathscr{F}_m$ for some $k\ge -1$ and each $n\in\mathbb{N}_0$.
Then we consider the following function:
\begin{eqnarray}\label{thm2-3}
&&G(t,x,\mu,u^k,\phi,\psi)\nonumber\\
&:=&\sum_{a\in A(x),b\in B(x)}\bigg[\int_{E_Y}\int_{\mathbb{R}}
\mathbb{I}_{[0,\lambda]}(r(x,y,a,b)t)(1-Q^X(t,E_X\mid x,y,a,b))\mu(dy,d\lambda)\nonumber\\
&&+\int_{E_X}\int_0^t u^{k}(t-\theta,x',\Upsilon(t,x,\mu,a,b,\theta,t-\theta,x'))
Q^X(d\theta,dx'\mid x,\mu^Y,a,b)\bigg]\phi(a)\psi(b).
\end{eqnarray}
Since (\ref{thm2-3}) are continuous on $A(x)$ and $B(x)$, and the convergence on $\mathscr{P}(A(x))$ and $\mathscr{P}(B(x))$ are weak convergence of probability measures, by Theorem 2.8.1 in \cite{ash2000probability},
(\ref{thm2-3}) are continuous on $\mathscr{P}(A(x))$ and $\mathscr{P}(B(x))$.
Moreover, by the finiteness of $A(x)$ and $B(x)$, we know that $\mathscr{P}(A(x))$ and $\mathscr{P}(B(x))$ are compact.
Thus, according to Theorem  A.2.3 in \cite{ash2000probability}, we derive that the function (\ref{thm2-3}) attains its supremum and infimum,
that is,
\begin{eqnarray*}
\inf_{\phi\in\mathscr{P}(A(x))}\sup_{\psi\in\mathscr{P}(B(x))}
G(t,x,\mu,u^k,\phi,\psi)=
\min_{\phi\in\mathscr{P}(A(x))}\max_{\psi\in\mathscr{P}(B(x))}
G(t,x,\mu,u^k,\phi,\psi).
\end{eqnarray*}
Moreover, by the minimax measurable selection theorem (see for instant Lemma 4.3 in \cite{nowak1984zero}), there exists an $(\phi_k,\psi_k)\in \Phi\times \Psi$ such that
\begin{eqnarray}\label{thm2-4}
u^{k+1}(t,x,\mu)&:=&
\sum_{a\in A(x),b\in B(x)}\bigg[\int_{E_Y}\int_{\mathbb{R}}
\mathbb{I}_{[0,\lambda]}(r(x,y,a,b)t)(1-Q^X(t,E_X\mid x,y,a,b))\mu(dy,d\lambda)\nonumber\\
&&+\int_{E_X}\int_0^t u^{k}(t-\theta,x',\Upsilon(t,x,\mu,a,b,\theta,t-\theta,x'))
Q^X(d\theta,dx'\mid x,\mu^Y,a,b)\bigg]\nonumber\\
&&~~\phi_k(a\mid t,x,\mu)\psi_k(b\mid t,x,\mu).
\end{eqnarray}
Thus, $u^{k+1}(t,x,\mu)\in\mathscr{F}_m$ and we have proved the claim above.
\\Next we prove the fact that for every $(t,x,\mu)\in \mathbb{R}_+\times E_X\times \mathscr{P}(E_X\times \mathbb{R})$ and $k\ge-1$,
\begin{eqnarray*}
u^{k+1}(t,x,\mu)\leq u^{k}(t,x,\mu).
\end{eqnarray*}
We will prove that by induction. When $k=-1$, for each $(t,x,\mu)\in \mathbb{R}_+\times E_X\times \mathscr{P}(E_X\times \mathbb{R})$,
\begin{eqnarray*}
&&u^0(t,x,\mu)\\
&:=&\inf_{\phi\in\mathscr{P}(A(x))}\sup_{\psi\in\mathscr{P}(B(x))}
\bigg\{\sum_{a\in A(x),b\in B(x)}
\bigg[\int_{E_Y}\int_{\mathbb{R}}
\mathbb{I}_{[0,\lambda]}(r(x,y,a,b)t)(1-Q^X(t,E_X\mid x,y,a,b))\mu(dy,d\lambda)\nonumber\\
&&+\int_{E_X}\int_0^t\int_{\mathbb{R}}\mathbb{I}_{[0,\infty)}(\lambda') \Upsilon(t,x,\mu,a,b,\theta,t-\theta,x')(E_Y,d\lambda')
Q^X(d\theta,dx'\mid x,\mu^Y,a,b)\bigg]\phi(a)\psi(b)\bigg\},\nonumber\\
&=&\inf_{\phi\in\mathscr{P}(A(x))}\sup_{\psi\in\mathscr{P}(B(x))}
\bigg\{\sum_{a\in A(x),b\in B(x)}
\bigg[\int_{E_Y}\int_{\mathbb{R}}
\mathbb{I}_{[0,\lambda]}(r(x,y,a,b)t)(1-Q^X(t,E_X\mid x,y,a,b))\mu(dy,d\lambda)\nonumber\\
&&+\int_{E_X}\int_0^t\int_{\mathbb{R}}\mathbb{I}_{[0,\infty)}(\lambda') \frac{\int_{E_Y}\int_{\mathbb{R}}\int_{E_Y}q(\theta,x',y'\mid x,y,a,b)\nu(dy')\delta_{\{\lambda-r(x,y,a,b)\theta\}}(d\lambda')
\mathbb{I}_{[0,\infty)}\mu(dy,d\lambda)}{\int_{E_Y}q^X(\theta,x'\mid x,y,a,b)\mu^Y(dy)}\\
&&
Q^X(d\theta,dx'\mid x,\mu^Y,a,b)\bigg]\phi(a)\psi(b)\bigg\}\nonumber\\
%&:=&\inf_{\phi\in\mathscr{P}(A(x))}\sup_{\psi\in\mathscr{P}(B(x))}
%\bigg\{\sum_{a\in A(x),b\in B(x)}
%\bigg[\int_{E_Y}\int_{\mathbb{R}}
%\mathbb{I}_{[0,\lambda]}(r(x,y,a,b)t)(1-Q^X(t,E_X\mid x,y,a,b))\mu(dy,d\lambda)\nonumber\\
%&&+\int_{E_X}\int_0^t\int_{E_Y}q^X(\theta,x'\mid x,y,a,b)\mu^Y(dy)\eta(dx')\\
%&&\times\frac{\int_{E_Y}\int_{\mathbb{R}}\int_{E_Y}q(\theta,x',y'\mid x,y,a,b)\nu(dy')\mathbb{I}_{[0,\infty)}(\lambda-r(x,y,a,b)\theta)
%\mu(dy,d\lambda)}
%{\int_{E_Y}q^X(\theta,x'\mid x,y,a,b)\mu^Y(dy)}
%\bigg]\phi(a)\psi(b)\bigg\},\nonumber\\
&=&\inf_{\phi\in\mathscr{P}(A(x))}\sup_{\psi\in\mathscr{P}(B(x))}
\{\sum_{a\in A(x),b\in B(x)}\int_{E_Y}\int_{\mathbb{R}}
\mathbb{I}_{[0,\lambda]}(r(x,y,a,b)t)\times 1\times \mu(dy,d\lambda)\phi(a)\psi(b)\}\\
&\leq&\int_{\mathbb{R}}
\mathbb{I}_{[0,\infty)}(\lambda)\mu^{\Lambda}(d\lambda)=u^{-1}(t,x,\mu),
\end{eqnarray*}
together with the definition of $\{u^k\}$ and the monotonicity of the operator $T$, we derive that for each $k\ge-1$, $u^{k+1}\leq u^{k}$. 
That is $\{u^k\}$ is the non-increasing sequence and thus it converges to some function $u^*\in\mathscr{F}_m$.
  
Next we will prove that $u^*(t,x,\mu)$ satisfies the Shapley equation, that is,
\begin{eqnarray}\label{thm2-5}
&&u^*(t,x,\mu)\nonumber\\
&=&\inf_{\phi\in\mathscr{P}(A(x))}\sup_{\psi\in\mathscr{P}(B(x))}
\bigg\{\sum_{a\in A(x),b\in B(x)}\bigg[\int_{E_Y}\int_{\mathbb{R}}
\mathbb{I}_{[0,\lambda]}(r(x,y,a,b)t)(1-Q^X(t,E_X\mid x,y,a,b))\mu(dy,d\lambda)\nonumber\\
&&+\int_{E_X}\int_0^t u^*(t-\theta,x',\Upsilon(t,x,\mu,a,b,\theta,t-\theta,x'))
Q^X(du,dx'\mid x,\mu^Y,a,b)\bigg]\phi(a)\psi(b)\bigg\}.
\end{eqnarray}
By $u^{k+1}(t,x,\mu)\leq u^{k}(t,x,\mu)$ and $\lim\limits_{k\to\infty}u^{k}(t,x,\mu)
:=u^*(t,x,\mu)$, we know that the right side of (\ref{thm2-5}) satisfying that for every $(t,x,\mu)\in \mathbb{R}_+\times E_X\times \mathscr{P}(E_X\times \mathbb{R})$ and $k\ge0$,
\begin{eqnarray*}
&&\inf_{\phi\in\mathscr{P}(A(x))}\sup_{\psi\in\mathscr{P}(B(x))}
\bigg\{\sum_{a\in A(x),b\in B(x)}\bigg[\int_{E_Y}\int_{\mathbb{R}}
\mathbb{I}_{[0,\lambda]}(r(x,y,a,b)t)(1-Q^X(t,E_X\mid x,y,a,b))\mu(dy,d\lambda)\nonumber\\
&&+\int_{E_X}\int_0^t u^*(t-\theta,x',\Upsilon(t,x,\mu,a,b,\theta,t-\theta,x'))
Q^X(du,dx'\mid x,\mu^Y,a,b)\bigg]\phi(a)\psi(b)\bigg\},\nonumber\\
&\leq&\inf_{\phi\in\mathscr{P}(A(x))}\sup_{\psi\in\mathscr{P}(B(x))}
\bigg\{\sum_{a\in A(x),b\in B(x)}\bigg[\int_{E_Y}\int_{\mathbb{R}}
\mathbb{I}_{[0,\lambda]}(r(x,y,a,b)t)(1-Q^X(t,E_X\mid x,y,a,b))\mu(dy,d\lambda)\nonumber\\
&&+\int_{E_X}\int_0^t u^k(t-\theta,x',\Upsilon(t,x,\mu,a,b,\theta,t-\theta,x'))
Q^X(du,dx'\mid x,\mu^Y,a,b)\bigg]\phi(a)\psi(b)\bigg\},\nonumber\\
&=&u^{k+1}(t,x,\mu),
\end{eqnarray*}
which implies that
\begin{eqnarray}\label{thm2-6}
&&\inf_{\phi\in\mathscr{P}(A(x))}\sup_{\psi\in\mathscr{P}(B(x))}
\bigg\{\sum_{a\in A(x),b\in B(x)}\bigg[\int_{E_Y}\int_{\mathbb{R}}
\mathbb{I}_{[0,\lambda]}(r(x,y,a,b)t)(1-Q^X(t,E_X\mid x,y,a,b))\mu(dy,d\lambda)\nonumber\\
&&+\int_{E_X}\int_0^t u^*(t-\theta,x',\Upsilon(t,x,\mu,a,b,\theta,t-\theta,x'))
Q^X(du,dx'\mid x,\mu^Y,a,b)\bigg]\phi(a)\psi(b)\bigg\}\nonumber\\
&\leq& u^*(t,x,\mu).
\end{eqnarray}
Now we show the reverse inequality. It follows from (\ref{thm2-2}) that
\begin{eqnarray}\label{thm2-7}
&&u^{k+1}(t,x,\mu)\nonumber\\
&=&\inf_{\phi\in\mathscr{P}(A(x))}\sup_{\psi\in\mathscr{P}(B(x))}
\bigg\{\sum_{a\in A(x),b\in B(x)}\bigg[\int_{E_Y}\int_{\mathbb{R}}
\mathbb{I}_{[0,\lambda]}(r(x,y,a,b)t)(1-Q^X(t,E_X\mid x,y,a,b))\mu(dy,d\lambda)\nonumber\\
&&+\int_{E_X}\int_0^t u^k(t-\theta,x',\Upsilon(t,x,\mu,a,b,\theta,t-\theta,x'))
Q^X(du,dx'\mid x,\mu^Y,a,b)\bigg]\phi(a)\psi(b)\bigg\}\nonumber\\
&\leq&\sup_{\psi\in\mathscr{P}(B(x))}
\bigg\{\sum_{a\in A(x),b\in B(x)}\bigg[\int_{E_Y}\int_{\mathbb{R}}
\mathbb{I}_{[0,\lambda]}(r(x,y,a,b)t)(1-Q^X(t,E_X\mid x,y,a,b))\mu(dy,d\lambda)\nonumber\\
&&+\int_{E_X}\int_0^t u^k(t-\theta,x',\Upsilon(t,x,\mu,a,b,\theta,t-\theta,x'))
Q^X(du,dx'\mid x,\mu^Y,a,b)\bigg]\phi(a)\psi(b)\bigg\}\nonumber\\
&=&
\sum_{a\in A(x),b\in B(x)}\bigg[\int_{E_Y}\int_{\mathbb{R}}
\mathbb{I}_{[0,\lambda]}(r(x,y,a,b)t)(1-Q^X(t,E_X\mid x,y,a,b))\mu(dy,d\lambda)\nonumber\\
&&+\int_{E_X}\int_0^t u^k(t-\theta,x',\Upsilon(t,x,\mu,a,b,\theta,t-\theta,x'))
Q^X(du,dx'\mid x,\mu^Y,a,b)\bigg]\phi(a)\psi^{k*}(b\mid t,x,\mu),\nonumber\\
\end{eqnarray}
for any $\phi\in\mathscr{P}(A(x))$, where the existence of $\psi^{k*}\in\mathscr{P}(B(x))$ (may depend on $\phi$) is guaranteed by the measurable selection theorem \citep{nowak1984zero}.
According to the compactness of $\mathscr{P}(B(x))$, without loss of generality, we suppose that $\psi^{k*}(\cdot\mid t,x,\mu)\to \psi^*(\cdot\mid t,x,\mu)\in\mathscr{P}(B(x))$.
Taking $k\to\infty$ in (\ref{thm2-7}), it can be obtained that
\begin{eqnarray*}
&&u^*(t,x,\mu)\nonumber\\
&\leq&\sum_{a\in A(x),b\in B(x)}\bigg[\int_{E_Y}\int_{\mathbb{R}}
\mathbb{I}_{[0,\lambda]}(r(x,y,a,b)t)(1-Q^X(t,E_X\mid x,y,a,b))\mu(dy,d\lambda)\nonumber\\
&&+\int_{E_X}\int_0^t u^*(t-\theta,x',\Upsilon(t,x,\mu,a,b,\theta,t-\theta,x'))
Q^X(du,dx'\mid x,\mu^Y,a,b)\bigg]\phi(a)\psi^*(b\mid t,x,\mu)\nonumber\\
&\leq&\sup_{\psi\in\mathscr{P}(B(x))}
\bigg\{\sum_{a\in A(x),b\in B(x)}\bigg[\int_{E_Y}\int_{\mathbb{R}}
\mathbb{I}_{[0,\lambda]}(r(x,y,a,b)t)(1-Q^X(t,E_X\mid x,y,a,b))\mu(dy,d\lambda)\nonumber\\
&&+\int_{E_X}\int_0^t u^*(t-\theta,x',\Upsilon(t,x,\mu,a,b,\theta,t-\theta,x'))
Q^X(du,dx'\mid x,\mu^Y,a,b)\bigg]\phi(a)\psi(b)\bigg\}
\end{eqnarray*}
for any $\phi\in\mathscr{P}(A(x))$.
Thus, we have
\begin{eqnarray}\label{thm2-8}
&&u^*(t,x,\mu)\nonumber\\
&\leq&\inf_{\phi\in\mathscr{P}(A(x))}\sup_{\psi\in\mathscr{P}(B(x))}
\bigg\{\sum_{a\in A(x),b\in B(x)}\bigg[\int_{E_Y}\int_{\mathbb{R}}
\mathbb{I}_{[0,\lambda]}(r(x,y,a,b)t)(1-Q^X(t,E_X\mid x,y,a,b))\mu(dy,d\lambda)\nonumber\\
&&+\int_{E_X}\int_0^t u^*(t-\theta,x',\Upsilon(t,x,\mu,a,b,\theta,t-\theta,x'))
Q^X(du,dx'\mid x,\mu^Y,a,b)\bigg]\phi(a)\psi(b)\bigg\},
\end{eqnarray}
which together with (\ref{thm2-6}) gets (\ref{thm2-5}). So we have proved that $u^*(t,x,\mu)$ satisfies the Shapley equation.
To verify (\ref{thm2-1}), it suffices to show $u^*(t,x,\mu)=V(t,x,\mu)$.
\\By (\ref{thm2-3}),
$G(t,x,\mu,u^*,\phi,\psi)$ are continuous on $\mathscr{P}(A(x))$ and $\mathscr{P}(B(x))$ and the compactness of $\mathscr{P}(A(x))$ and $\mathscr{P}(B(x))$. Moreover, $G(t,x,\mu,u^*,\phi,\psi)$ is convex in $\phi\in\mathscr{P}(A(x))$ and concave in $\psi\in\mathscr{P}(B(x))$.
Hence, Fan's minimax theorem \citep{fan1953minimax} and the measurable selection theorem \citep{nowak1984zero} ensure that for each $(t,x,\mu)\in \mathbb{R}_+\times E_X\times \mathscr{P}(E_X\times \mathbb{R})$ and $k\ge0$,
\begin{eqnarray}
&&u^*(t,x,\mu)\nonumber\\
&=&\inf_{\phi\in\mathscr{P}(A(x))}\sup_{\psi\in\mathscr{P}(B(x))}
\bigg\{\sum_{a\in A(x),b\in B(x)}\bigg[\int_{E_Y}\int_{\mathbb{R}}
\mathbb{I}_{[0,\lambda]}(r(x,y,a,b)t)(1-Q^X(t,E_X\mid x,y,a,b))\mu(dy,d\lambda)\nonumber\\
&&+\int_{E_X}\int_0^t u^*(t-\theta,x',\Upsilon(t,x,\mu,a,b,\theta,t-\theta,x'))
Q^X(du,dx'\mid x,\mu^Y,a,b)\bigg]\phi(a)\psi(b)\bigg\}\nonumber\\
&=&\min_{\phi\in\mathscr{P}(A(x))}\max_{\psi\in\mathscr{P}(B(x))}
\bigg\{\sum_{a\in A(x),b\in B(x)}\bigg[\int_{E_Y}\int_{\mathbb{R}}
\mathbb{I}_{[0,\lambda]}(r(x,y,a,b)t)(1-Q^X(t,E_X\mid x,y,a,b))\mu(dy,d\lambda)\nonumber\\
&&+\int_{E_X}\int_0^t u^*(t-\theta,x',\Upsilon(t,x,\mu,a,b,\theta,t-\theta,x'))
Q^X(du,dx'\mid x,\mu^Y,a,b)\bigg]\phi(a)\psi(b)\bigg\}\nonumber\\
&=&\max_{\psi\in\mathscr{P}(B(x))}\min_{\phi\in\mathscr{P}(A(x))}
\bigg\{\sum_{a\in A(x),b\in B(x)}\bigg[\int_{E_Y}\int_{\mathbb{R}}
\mathbb{I}_{[0,\lambda]}(r(x,y,a,b)t)(1-Q^X(t,E_X\mid x,y,a,b))\mu(dy,d\lambda)\nonumber\\
&&+\int_{E_X}\int_0^t u^*(t-\theta,x',\Upsilon(t,x,\mu,a,b,\theta,t-\theta,x'))
Q^X(du,dx'\mid x,\mu^Y,a,b)\bigg]\phi(a)\psi(b)\bigg\}\nonumber\\
&=&\max_{\psi\in\mathscr{P}(B(x))}
\bigg\{\sum_{a\in A(x),b\in B(x)}\bigg[\int_{E_Y}\int_{\mathbb{R}}
\mathbb{I}_{[0,\lambda]}(r(x,y,a,b)t)(1-Q^X(t,E_X\mid x,y,a,b))\mu(dy,d\lambda)\nonumber\\
&&+\int_{E_X}\int_0^t u^*(t-\theta,x',\Upsilon(t,x,\mu,a,b,\theta,t-\theta,x'))
Q^X(du,dx'\mid x,\mu^Y,a,b)\bigg]\phi^*(a\mid t,x,\mu)\psi(b)\bigg\}\label{thm2-10}\\
&=&\min_{\phi\in\mathscr{P}(A(x))}
\bigg\{\sum_{a\in A(x),b\in B(x)}\bigg[\int_{E_Y}\int_{\mathbb{R}}
\mathbb{I}_{[0,\lambda]}(r(x,y,a,b)t)(1-Q^X(t,E_X\mid x,y,a,b))\mu(dy,d\lambda)\nonumber\\
&&+\int_{E_X}\int_0^t u^*(t-\theta,x',\Upsilon(t,x,\mu,a,b,\theta,t-\theta,x'))
Q^X(du,dx'\mid x,\mu^Y,a,b)\bigg]\phi(a)\psi^*(b\mid t,x,\mu)\bigg\}\label{thm2-11}\\
&=&\sum_{a\in A(x),b\in B(x)}\phi^*(a\mid t,x,\mu)\psi^*(b\mid t,x,\mu)\bigg[\int_{E_Y}\int_{\mathbb{R}}
\mathbb{I}_{[0,\lambda]}(r(x,y,a,b)t)(1-Q^X(t,E_X\mid x,y,a,b))\nonumber\\
&&\times\mu(dy,d\lambda)+\int_{E_X}\int_0^t u^*(t-\theta,x',\Upsilon(t,x,\mu,a,b,\theta,t-\theta,x'))
Q^X(du,dx'\mid x,\mu^Y,a,b)\bigg]\label{thm2-12}
\end{eqnarray}
Under Assumption \ref{ass1}, by (\ref{thm2-11}) and Theorem \ref{Comparison}(a), we can obtain that
\begin{eqnarray*}
u^*(t,x,\mu)\leq F^{\pi^{1},\psi^{*}}(t,x,\mu)~\forall~ (t,x,\mu)\in \mathbb{R}_+\times E_X\times \mathscr{P}(E_X\times \mathbb{R})~ and ~\pi^1\in\Pi^1,
\end{eqnarray*}
%For fixed $\psi^*$, 
which implies that
\begin{eqnarray}\label{thm2-13}
u^*(t,x,\mu)
\leq \inf_{\pi^{1}\in\Pi^1}F^{\pi^{1},\psi^{*}}(t,x,\mu)
\leq \sup_{\pi^{2}\in\Pi^2}\inf_{\pi^{1}\in\Pi^1}F^{\pi^{1},\pi^{2}}(t,x,\mu)
=\underline{V}(t,x,\mu).
\end{eqnarray}
Using Theorem \ref{Comparison}(b) and (\ref{thm2-10}), a similar discussion shows that
\begin{eqnarray}\label{thm2-14}
u^*(t,x,\mu)
\ge\sup_{\pi^{2}\in\Pi^2}F^{\pi^{1},\psi^{*}}(t,x,\mu)
\ge \inf_{\pi^{1}\in\Pi^1}\sup_{\pi^{2}\in\Pi^2}F^{\pi^{1},\pi^{2}}(t,x,\mu)
=\overline{V}(t,x,\mu).
\end{eqnarray}
Hence, $u^*(t,x,\mu)=\underline{V}(t,x,\mu)=\overline{V}(t,x,\mu)=V(t,x,\mu)$. 
\\(b). This is a direct results by (\ref{thm2-13})-(\ref{thm2-14}).
\end{proof}
\section{Conclusion}\label{sec5}
In this paper, we study the optimization of
POSMGs with a risk probability criterion. The objective is  to find a Nash equilibrium to
minimize the probability that the accumulated rewards of player 1 (i.e., the incurred costs of player 2)  fall short of a given goal during a finite horizon. Extending the state space by the joint conditional distribution of current unobserved state and the remaining  goal, the POSMGs can be reformulated with the corresponding completely observable problem.
Then we establish a comparison theorem and derive the Shapley equation, and based on which, we prove the existence of the value of the game and a Nash equilibrium.

One of the future research topics is
to study risk criteria in   the framework of N-person nonzero-sum stochastic games with the partially observable perspective. The objective is to minimize the probability of total rewards that do not exceed the profit goal for each player and find the existences of the Nash equilibrium, which is not reported in the literature yet.

\end{document}